\documentclass{amsart}

\usepackage{amsmath, amssymb, amsbsy}
\usepackage{array}
\usepackage{graphpap, color, paralist,xcolor}
\usepackage[mathscr]{eucal}
\usepackage[pdftex]{graphicx}
\usepackage[pdftex,colorlinks,backref=page,citecolor=blue]{hyperref}
\usepackage{pifont}

\usepackage{mathtools}
\usepackage{cleveref}
\usepackage{float}
\usepackage{crossreftools}
\DeclarePairedDelimiter{\ceil}{\lceil}{\rceil}
\newcounter{bullet}

\usepackage{setspace}
\setdisplayskipstretch{1}

\usepackage{geometry}
\usepackage{comment}
\newtheorem{thm}{Theorem}[section]
\newtheorem{prop}[thm]{Proposition}
\newtheorem{cor}[thm]{Corollary}
\newtheorem{lem}[thm]{Lemma}

\theoremstyle{definition}

\crefname{lem}{lemma}{lemmas}

\newcommand{\go}{\omega}

\newcommand{\RR}{\mathbb{R}}
\newcommand{\Z}{\mathbb{Z}}

\newcommand{\cH}{\mathcal{H} }

\newcommand{\cL}{\mathcal{L} }

\newcommand{\cU}{\mathcal{U} }

\newcommand{\bff}{\mathbf{f}}

\newcommand{\beq}[1]{\begin{equation}\label{#1}}
\newcommand{\enq}[0]{\end{equation}}

\newcommand{\eps}{\epsilon}
\newcommand{\ep}{\varepsilon}

\newcommand{\gl}[0]{\lambda}

\newcommand{\nin}[0]{\noindent}

\newcommand{\sub}[0]{\subseteq}

\newcommand{\ra}[0]{\rightarrow}

\newcommand{\pr}[0]{\mathbb{P}}

\newcommand{\ext}[0]{\mbox{\rm{ex}}}

\newcommand{\Hom}[0]{\mbox{\rm{Hom}}}
\newcommand{\Lip}[0]{\mbox{\rm{Lip}}}
\newcommand{\dist}[0]{\mbox{\rm{dist}}}

\newcommand{\inbound}{\overline{\partial}}
\newcommand{\outbound}{\partial}

\newcommand{\floor}[1]{\left\lfloor#1\right\rfloor}
\renewcommand{\ceil}[1]{\left\lceil#1\right\rceil}

\begin{document}

\title{Lipschitz functions on weak expanders}

\author[R. A. Krueger]{Robert A. Krueger}
\address{Department of Mathematical Sciences, Carnegie Mellon University}
\email{rkrueger@andrew.cmu.edu}

\author[L. Li]{Lina Li}
\address{Department of Mathematics, Iowa State University}
\email{linali@iastate.edu}

\author[J. Park]{Jinyoung Park}
\address{Department of Mathematics, Courant Institute of Mathematical Sciences, New York University}
\email{jinyoungpark@nyu.edu}

\begin{abstract}
Given a connected finite graph $G$, an integer-valued function $f$ on $V(G)$ is called $M$-Lipschitz if the value of $f$ changes by at most $M$ along the edges of $G$. In 2013, Peled, Samotij, and Yehudayoff showed that random $M$-Lipschitz functions on graphs with sufficiently good expansion typically exhibit small fluctuations, giving sharp bounds on the typical range of such functions, assuming $M$ is not too large. We prove that the same conclusion holds under a relaxed expansion condition and for larger $M$, (partially) answering questions of Peled et al. Our techniques involve a combination of Sapozhenko's graph container methods and entropy methods from information theory.
\end{abstract}

\maketitle

\section{Introduction}\label{sec.intro}

We consider two kinds of Lipschitz functions on graphs: $\Z$-homomorphisms and (integer-valued) $M$-Lipschitz functions. Throughout the paper, $G=(V,E)$ is a connected finite graph. A \emph{$\Z$-homomorphism} from $G$ is a function $f:V(G) \ra \Z$ such that $|f(u)-f(v)|=1$ for all $\{u, v\} \in E(G)$. Note that a $\Z$-homomorphism on $G$ exists if and only if $G$ is bipartite. The simple random walk on $\Z$ (up to time $n$) is a uniform random $\Z$-homomorphism from a path of length $n$ to $\Z$ which sends an endpoint of the path to $0$. The uniform measure on the set of all $\mathbb Z$-homomorphisms from a general graph $G$ sending some fixed vertex $v_0$ to $0$ generalizes this notion, and hence is often called a \emph{$G$-indexed random walk} (e.g.~\cite{benjamini2000random,benjamini1994tree}). A natural extension of $\Z$-homomorphisms are \emph{$M$-Lipschitz functions}, integer-valued functions on $V(G)$ that satisfy $|f(u)-f(v)| \le M$ for all $\{u,v\} \in E(G)$.

Historically, of particular interest has been the (distribution of) the range of a random Lipschitz function from a given graph $G$. We define the \emph{range of $f$} to be
\[R(f):=\max_{v \in V(G)}f(v)-\min_{v \in V(G)}f(v)+1.\]
In order to make sense of a ``uniformly random'' Lipschitz function on a finite graph, it is natural to enforce a ``one-point boundary condition.'' Given $G$ and $v_0 \in V(G)$, let
\[\Hom_{v_0}(G)=\{f:V(G) \ra \Z \text{ such that $f(v_0)=0$ and $|f(u)-f(v)|=1$ $ \forall \{u,v\} \in E(G)$}\},\]
and
\[\Lip_{v_0}(G;M)=\{f:V(G) \ra \Z \text{ such that $f(v_0)=0$ and $|f(u)-f(v)|\le M$ $ \forall \{u,v\} \in E(G)$}\}.\]
Denote by $x \in_R X$ a uniformly random element $x$ of $X$. Since we assume throughout that $G$ is connected and finite, both $\Hom_{v_0}(G)$ and $\Lip_{v_0}(G;M)$ are finite. Moreover, the choice of $v_0$ is not particularly important: for any $v_0, v_1 \in V(G)$, there is a natural bijection between $\Hom_{v_0}(G)$ and $\Hom_{v_1}(G)$ (and between $\Lip_{v_0}(G;M)$ and $\Lip_{v_1}(G;M)$) using translations of $\Z$. In particular, the distribution of $R(f)$ for $f \in_R \Hom_{v_0}(G)$ (and for $f \in_R \Lip_{v_0}(G;M)$) is independent of the choice of $v_0$. 

Benjamini, H\"aggstr\"om, and Mossel~\cite{benjamini2000random} investigated the distribution of $R(f)$ for $f \in_R\Hom_{v_0}(G)$ with a motivating question of whether concentration inequalities for typical Lipschitz functions are stronger than those which hold for all Lipschitz functions. Among many results, one class of graphs that they studied was $T_{n,d}$, the $n$-level $d$-regular tree ``wired at'' the $n$-th level (meaning all leaves on the last level are connected to an additional single vertex), inspired by a branching random walk (see e.g. \cite{asmussen1983branching} or \cite{ney1991branching}) on the event that all particles occupy the same location at time $n+1$. Benjamini et al.\ showed that for $f \in_R \Hom_{v_0}(T_{n,d})$, $R(f)=\Theta_d(\log n)$ w.h.p.\footnote{with high probability --- that is, there exists $c_d$ and $C_d$ such that as $n\to \infty$, the probability that $R(f) \in [c_d \log n, C_d \log n]$ tends to $1$.} This is in surprising contrast to the maximum possible value of $R(f)$, which is trivially the diameter $n+1$ of $T_{n,d}$. (The \emph{diameter} of a graph $G$, denoted $\text{diam}(G)$, is $\max_{u,v \in V(G)} \dist_G(u,v)$, where $\dist_G(u,v)$ is the length of a shortest path between $u$ and $v$.) More generally, Benjamini and Schechtman~\cite{benjamini2000upper} showed that for any $n$-vertex $G$ and $f \in_R \Hom_{v_0}(G)$, the expectation of $R(f)$ is $O(\sqrt n)$. (A similar result for random 1-Lipschitz functions was proven by Loebl, Ne{\v{s}}et{\v{r}}il, and Reed~\cite{loebl2003note}). The conjecture that the maximum of $\mathbb E R(f)$ among all $n$-vertex graphs is achieved by a path remains open (\cite{benjamini2000upper,loebl2003note}).

Another example of the stark contrast between typical and generic $\Z$-homomorphisms is the hypercube. It was conjectured in \cite{benjamini2000random} that the range of a typical $\Z$-homomorphism on the $d$-dimensional hypercube, $Q_d$, is $o(d)$ (note that $\text{diam}(Q_d) = d$). This conjecture was proven in a very strong way by Kahn~\cite{kahn2001range}, who showed that the typical range is actually $O(1)$. Even further, Galvin~\cite{galvin2003homomorphisms} showed that, w.h.p., a random $\Z$-homomorphism on $Q_d$ takes at most 5 values. Peled~\cite{peled2017high} vastly generalized Galvin's work, obtaining results for a general class of tori including $\Z_n^d=(\Z/n\Z)^d$ (where $n$ can be large with respect to $d$) and providing strong bounds on $\mathbb P(R(f) \ge k)$ for arbitrary $k$, for both random $\Z$-homomorphisms and 1-Lipschitz functions. 

These works suggest that typical Lipschitz functions on ``highly connected'' graphs tend to have small fluctuations. Motivated by this, Peled, Samotij, and Yehudayoff~\cite{peled2013lipschitz} initiated the study of the range of typical Lipschitz functions (specifically $\mathbb{Z}$-homomorphisms and $M$-Lipschitz functions) on expander graphs (see~\eqref{def.expander} for the formal definition of expander graph). Expander graphs are historically of great importance (see e.g.~\cite{hoory2006expander}), and they form a wide class of graphs which include ``most graphs'' (that is, random graphs~\cite{friedman2003proof}). In the current context, expander graphs are natural test case for the aforementioned ``highly connected''--``small fluctuations'' paradigm. Expander graphs may be locally tree-like, suggesting that Lipschitz functions may have large ranges. However, like Bejamini et al.~\cite{benjamini2000random}, Peled et al.~\cite{peled2013lipschitz} show that the global connectivity features of expander graphs win out over these local features, causing typical Lipschitz functions to have small fluctuations. Our main result considerably weakens the expansion condition needed for the conclusions of Peled et al.

Before we precisely state our results, we draw attention to the distinction between $\Z$-homomorphisms and $1$-Lipschitz functions. Every $\Z$-homomorphism is $1$-Lipschitz, but not vice-versa; $\Lip_{v_0}(G;1)$ is usually far bigger than $\Hom_{v_0}(G)$. It is plausible that $\Z$-homomorphisms and $1$-Lipschitz functions behave very similarly with respect to their ranges and typical fluctuations, yet the following conjecture of Benjamini, Yadin, and Yehudayoff~\cite{benjamini2007random} still appears open: for every bipartite graph $G$ of maximum degree at most $d$, if $f \in_R \Hom_{v_0}(G)$ and $g \in_R \Lip_{v_0}(G;1)$, then $\mathbb E[R(f)]/\mathbb E[R(g)]=\Theta_d(1)$. Nonetheless, for expander graphs, our proof techniques that we present for $M$-Lipschitz functions almost immediately\footnote{We choose to not include the corresponding results and proofs for $\Z$-homomorphisms because doing so would require an entirely new set of definitions: $\Z$-homomorphisms only exist on bipartite graphs, thus requiring the notion of a bipartite expander graph. See~\cite{peled2013lipschitz} for a treatment of both $\Z$-homomorphisms and $M$-Lipschitz functions. We plan to revisit $\Z$-homomorphisms in future work~\cite{KLP2} discussing general homomorphisms.} apply to $\Z$-homomorphisms, and thus we will restrict our attention to the more varied setting of $M$-Lipschitz functions in the current paper. 

Consideration of $M$-Lipschitz functions is a natural bridge from $\Z$-homomorphisms to \emph{$\RR$-valued Lipschitz functions on $G$}, functions $f: V(G) \to \mathbb{R}$ satisfying $|f(u)-f(v)|\le 1$ for all $\{u,v\} \in E(G)$. Indeed, if $f \in_R \Lip_{v_0}(G;M)$, then as $M \ra \infty$, $f/M$ converges in distribution to a uniformly random $\RR$-valued Lipschitz function on $G$ with $f(v_0)=0$ (see e.g. \cite[page 8]{peled2013grounded}). While our results do not apply for $M$ unbounded by the graph $G$, we believe this is an important and interesting direction for future research --- see Section~\ref{sec:open} for further discussion.

Now we precisely state our results. Following \cite{peled2013lipschitz}, we say an $n$-vertex, $d$-regular graph $G$ is a \emph{$\lambda$-expander} if
\beq{def.expander}
\left|e(S,T)-\frac{d}{n}|S||T|\right|\le\gl \sqrt{|S||T|} \quad \text{for all $S, T \sub V(G)$} ,
\enq
where $e(S,T)$ is the number of edges with one endpoint in $S$ and the other in $T$, counted twice if both endpoints are in $S \cap T$. The definition of a $\lambda$-expander is inspired by the expander mixing lemma \cite{alon1988explicit}, which states that if $G$ is a $d$-regular graph whose eigenvalues (of its adjacency matrix) are either $d$ or bounded in absolute value by $\lambda$, then $G$ is a $\lambda$-expander. It follows from the definition that every $d$-regular graph is a $d$-expander, so it is only meaningful to consider $\gl \le d$. Furthermore, plugging in $S=\{v\}$ and $T=N(v)$ (the neighborhood of $v$, that is, the set of all vertices adjacent to $v$) yields the conclusion that $d$-regular $\lambda$-expanders must satisfy 
\beq{lambdalb}\lambda \geq \sqrt{d}(1-d/n).
\enq
Note that the smaller $\lambda$ is, the stronger the expansion given by ~\eqref{def.expander} is. Informally, we call $d$-regular $\lambda$-expanders `optimal' if $\lambda$ is on the order of $\sqrt{d}$, and we call them `weak' if $\lambda$ is on the order of $d$. We note that most $d$-regular graphs are optimal expanders, in the sense that for every $d \geq 3$ and every $\eps>0$, a random $d$-regular graph is a $(2\sqrt{d-1}+\eps)$-expander w.h.p.~\cite{friedman2003proof}.

As a first step to show the 'flatness' of typical $M$-Lipschitz functions on expanders, Peled et al.\ showed that \emph{every} $M$-Lipschitz function takes values in a set of $M+1$ consecutive integers on all but a small fraction of its vertices, even for weak expanders:
\begin{lem}[Lemma 1.1, \cite{peled2013lipschitz}]\label{lem.Lips.GS}
Let $G$ be an $n$-vertex, $d$-regular $\gl$-expander. For every $M$-Lipschitz $f$ on $G$, there exists $k \in \mathbb{Z}$ such that
\beq{def.phase}
\left|\{v \in V(G) : f(v) \not\in \{k,k+1,\dots,k+M\}\}\right|\le\frac{2\gl}{d} n .
\enq
\end{lem}
Roughly, the above lemma follows from the fact that the sets $\{v \in V(G) : f(v)\leq k\}$ and $\{v \in V(G) : f(v)>k+M\}$ cannot have an edge between them in $G$ (since $f$ is $M$-Lipschitz), while expanders do not have large bipartite `holes' (see Proposition~\ref{prop:conn}).

Below is our main result, which is an improvement of \cite[Corollary 1.3]{peled2013lipschitz}. For the rest of this section, logarithms are of base 2.

\begin{thm}\label{thm.flat}
There exist universal constants $c,C,c',C'>0$ such that the following holds. Let $M \in \mathbb Z^+$, and let $G$ be a connected $n$-vertex, $d$-regular $\gl$-expander with $\lambda \leq d/5 \leq cn$, where
\beq{eq.Mld}
M \leq \min\left\{ c \frac{d^{3/2}}{\lambda\log d}, (\log n)^C \right\}.
\enq
Fix $v_0 \in V(G)$, and let $f$ be chosen uniformly at random from $\Lip_{v_0}(G;M)$. Then
\[ \pr\left( R(f) \geq C'M \frac{\log\log(n)}{\log(d/\lambda)} + 2(M+1) \right) \leq n^{-c'} .\]
Furthermore, for every $v \in V(G)$, we have $\mathrm{Var}(f(v)) = O\left((M\ceil{\log(M)/\log(d/2\lambda)})^2\right)$ where the implicit constant is universal.
\end{thm}
\nin By adjusting the constants, the above theorem is vacuous if $d$ or $n$ are small, so we assume that $d$ and $n$ are large. For the future reference, we record that the assumption that $d/5 \le cn$ together with~\eqref{lambdalb} yields
\beq{lambdalb1} \gl \ge \sqrt d/2.\enq

Peled et al.~\cite{peled2013lipschitz} proved Theorem~\ref{thm.flat} under the assumption
\beq{PSYcond} \lambda \leq \frac{d}{32(M+1)\log(9Md^2)}, \enq
which, in particular, requires that $\lambda = O(d/\log(d))$. 
Peled et al.\ asked how tight is~\eqref{PSYcond} --- in particular, do their results continue to hold when $\lambda$ is slightly less than $d$, or for larger $M$? We partially answer these questions with \Cref{thm.flat}: the conclusion holds even for weak expanders (that is, $d/\gl=O(1)$), and furthermore, for optimal expanders (that is, $\lambda = \Theta(\sqrt{d})$), we allow $M$ to be as big as $d/\log (d)$ (in~\cite{peled2013lipschitz}, $M = O(\sqrt{d}/\log d)$).

The $M \log\log(n)$ in \Cref{thm.flat} is best possible, that is to say, typical $M$-Lipschitz functions on expanders (subject to~\eqref{eq.Mld}) have range $\Theta_d(M \log\log(n))$. We explain this briefly here. Since $d$-regular $\lambda$-expanders have vertex expansion at least roughly $(d/\lambda)^2$ (meaning that for ``not-too-large'' sets $A$ of vertices, $|N(A)| \gtrsim (d/\lambda)^2 |A|$), there are at least roughly $(d/\lambda)^{2t}$ vertices within distance $t$ of any given vertex, if $t$ is not too large. This implies that the diameter of an $n$-vertex, $d$-regular $\lambda$-expander is $\Theta_{d,\lambda}(\log n)$, so \Cref{thm.flat} indeed shows that typical Lipschitz functions have range much smaller than the maximum possible range. It was noted in~\cite{peled2013lipschitz} that one may apply the techniques of Benjamini, Yadin, and Yehudayoff \cite{benjamini2007random} to obtain a matching lower bound on the range of $f$, that is, $\mathbb E[R(f)] \ge O_{d,\gl}(M\log \log n)$, if $M$ is relatively smaller than $n$. Roughly, the idea is that we can greedily find $\Theta(n/\log n)$ many pairwise disjoint balls of radius $\Theta(\log\log n)$ (which must have $\Theta(\log n)$ vertices) and then on at least one of them the function is likely to take an extreme value. We do not explore this further.

The uniform variance bound in Theorem~\ref{thm.flat} is obtained without much extra work. Such a bound is indicative of a ``localization'' phenomenon common for ``height functions'' (of which Lipschitz functions are an example) on high-dimensional graphs. This is in contrast to, for example, $M$-Lipschitz functions on the $n$-vertex path, where the variance grows with the distance from the boundary condition. See for example~\cite{chandgotia2021delocalization, lammers2024height, milos2015delocalization, peled2017high, peled2014random} for other models where this ``localization-delocalization'' phenomenon is present.

In our efforts to prove \Cref{thm.flat}, it turns out that rather than using a classical \emph{one-point boundary condition} for our Lipschitz functions (that is, setting $f(v_0)=0$), it is more convenient, and (in our opinion) also natural, to study an alternative conditioning, which we introduce as the \emph{fixed ground state condition}. Given $G$ an $n$-vertex, $d$-regular $\gl$-expander, denote by $\Lip(G;M)$ the (infinite) collection of all $M$-Lipschitz functions on $G$ (without the boundary condition $f(v_0)=0$). Lemma~\ref{lem.Lips.GS} establishes as a \emph{ground state} of $f \in \Lip(G;M)$ to be an interval $\{k,\dots,k+M\}$ that satisfies \eqref{def.phase}: it is helpful to think of $f \in \Lip(G;M)$ as mostly residing in $\{k,\dots,k+M\}$ for such a $k \in \mathbb{Z}$, with rare fluctuations outside of that range. Let
\[ \Lip^*_k(G;M) = \left\{ f \in \Lip(G;M) : \left|\left\{ v \in V(G) : f(v) \not\in \{k,k+1,\dots,k+M\} \right\}\right| \leq \frac{2\lambda}{d} n \right\} \]
be the collection of all Lipschitz functions with ground state $\{k, \ldots, k+M\}$. Note that $f \in \Lip(G;M)$ can be contained in multiple different $\Lip_k^*(G;M)$. For $v \in V(G)$ and $t \geq 0$, we denote by $B(v,t)$ the \emph{ball of radius $t$ centered at $v$}, that is, $B(v,t) = \{u \in V(G) : \dist_G(u,v) \leq t\}$. The following is our main technical result, which we use to prove Theorem~\ref{thm.flat}.
\begin{thm}\label{MT.nonbip}
There exists a universal constants $c,C>0$ such that the following holds.
Let $M \in \mathbb Z^+$, and let $G$ be a connected $n$-vertex $d$-regular $\gl$-expander with $\lambda \leq d/5 \leq cn$, where \eqref{eq.Mld} holds.
Fix $k \in \mathbb{Z}$, and let $f$ be chosen uniformly at random from $\Lip^*_k(G;M)$. Then for every integer $t \geq 2$ and every $v \in V(G)$,
\[ \pr( f(v) > k + tM+1) \le \exp_2\left(-\frac{|B(v,t-1)|}{5M}\right) .\]
\end{thm}
As noted before, in expander graphs, $|B(v,t)|$ is exponential in $t$, so the tail probability in \Cref{MT.nonbip} is double exponentially small. We do not try to optimize the constants; the bound $\lambda \leq d/5$ is not crucial, but we cannot take $\lambda$ arbitrarily close to $d$; see Section~\ref{sec:open}. Note that by symmetry, \Cref{MT.nonbip} applies equally well to fluctuations below the ground state. With an application of Bayes' rule, we can modify the distribution in Theorem~\ref{MT.nonbip} from $\Lip^*_k(G;M)$ (fixed ground state) to $\Lip_{v_0}(G;M)$ (fixed one-point boundary condition) with some negligible factors. In particular, this gives bounds on the probability that  $f \in_R \Lip_{v_0}(G;M)$ satisfies~\eqref{def.phase} for a particular $k$; see \Cref{sec.GS}.

The two main ingredients of our proof are entropy techniques and graph containers (see \Cref{pf.overview} for a proof overview), both of which have proven useful in asymptotic enumeration problems and statistical physics problems (e.g.~\cite{galvin2004phase, peled2023rigidity}). In particular, we prove a version of Sapozhenko's graph container lemma~\cite{sapozhenko1987number} with improved dependency on parameters for expander graphs which we believe is of independent interest; see \Cref{sec:containers}. For recent work using this container method which is related to our approach here, see e.g.~\cite{jenssen2023homomorphisms, kahn2020number, kahn2022number, li2023number}; see also~\cite{kleitman1982number,sapozhenko1987number} for some early applications of containers. Our integration of entropy and containers is inspired by \cite{li2023number}, which studied colorings on the middle two layers of the hypercube, but there are numerous technical differences and additional obstacles to overcome.

The rest of the paper is organized as follows. In Section~\ref{sec:prelim}, we fix our notation and give preliminary results about expanders and entropy. In Section~\ref{sec:containers}, we establish the containers that we use in our proof of Theorem~\ref{MT.nonbip}, which is given in Section~\ref{sec:main}. An overview of the proof of Theorem~\ref{MT.nonbip} is given in Section~\ref{pf.overview}. In Section~\ref{sec.GS}, we use Theorem~\ref{MT.nonbip} to prove Theorem~\ref{thm.flat}. Finally, in Section~\ref{sec:open}, we discuss future work and open problems.

\section{Notation and preliminaries}\label{sec:prelim}

We summarize some notation and preliminaries here. In Section~\ref{sec:prelim:notation}, we give our notation and a few standard lemmas; Section~\ref{sec:prelim:expanders} provides some basic propositions regarding $\lambda$-expanders; and Section~\ref{sec:prelim:entropy} collects basic properties of the binary entropy function of discrete random variables and a formulation of Shearer's Lemma.

\subsection{Notation and terminology}\label{sec:prelim:notation}

We use $\log$ for $\log_2$, $\ln$ for $\log_e$, and $\exp_2(x)$ for $2^x$. For integers $a,b$ with $a \leq b$, we write $[a,b]$ for $\{a,a+1,\dots,b-1,b\}$, and $[a]$ for $\{1,\dots,a\}$ if $a$ is positive.

Throughout $G$ is a connected simple $d$-regular $n$-vertex graph.
For $X \sub V(G)$, we denote the \emph{neighborhood} of $X$ by $N(X) := \{v \in V(G): \exists u \in X \text{ such that } uv \in E(G)\}$; when $X = \{v\}$, we simply write $N(v)$.
For $Y\subseteq V(G)$, we let $d_Y(v):=|Y\cap N(v)|$; when $Y = V(G)$, we simply write $d(v) = |N(v)|$ for the \emph{degree} of $v$.
The \emph{ball of radius $t$ about $X$}, denoted by $B(X,t)$, is $\{v \in V(G) : \exists u \in X \text{ such that } \dist_G(u,v) \leq t\}$; when $X = \{v\}$, we simply write $B(v,t)$.
For convenience, we let $X^+ := B(X,1) = X \cup N(X)$. The \emph{outer boundary} of $X$, denoted by $\outbound(X)$, is $N(X) \setminus X$, while the \emph{inner boundary} of $X$, denoted by $\inbound(X)$, is $\outbound(V(G)\setminus X)$. The \emph{interior} of $X$, denoted by $X^\circ$, is $X \setminus \inbound(X) = V(G) \setminus (V(G)\setminus X)^+$.

For $X \subseteq V(G)$, $G[X]$ denotes the subgraph of $G$ induced on $X$. For $k \in \mathbb Z^+$, we use $G^k$ for the \emph{$k^{\text{th}}$ power of $G$}, which is the graph on $V(G)$ where two distinct vertices are adjacent if and only if their distance in $G$ is at most $k$. We say that $X \subseteq V(G)$ is \emph{$k$-linked in $G$} if $G^k[X]$ is connected. When the underlying graph is clear, we simply say that $X$ is $k$-linked. 
For any $Y\subseteq V(G)$, we call the maximal $k$-linked subsets of $Y$ the \emph{$k$-linked components of $Y$}.

For $X, Y \subseteq V(G)$, we say that $Y$ \emph{covers} $X$ if $X \subseteq Y^+$. If $X$ covers $Y$ and $Y$ covers $X$, we say that $Y$ \emph{mutually covers} $X$. In the following proposition, we observe that linked-ness transfers to mutual covers, omitting the proof.

\begin{prop}\label{prop:cover}
Let $X, Y$ be sets of vertices of a graph $G$ such that $Y$ mutually covers $X$. If $X$ is $k$-linked, then $Y$ is $(k+2)$-linked.
\end{prop}

We also need a standard lemma bounding the number of (rooted) $k$-linked sets of a graph. The following is a consequence of the fact that the infinite $\Delta$-branching rooted tree contains precisely
\[\frac{\binom{\Delta m}{m}}{(\Delta-1)m+1} \le (e\Delta)^{m-1}\]
rooted subtrees with $m$ vertices (see for example \cite[p.\ 396, Ex.11]{knuth2005art}). 
\begin{lem}\label{lem:numlinkset}
If $G$ is a graph with maximum degree $\Delta$, then the number of $m$-vertex
subsets of $V(G)$ which contain a fixed vertex and induce a connected subgraph is at most $(e\Delta)^{m-1}$.
\end{lem}

\subsection{Properties of expanders}\label{sec:prelim:expanders}

Recall from \eqref{def.expander} the definition of a $\gl$-expander. In this section, we collect some basic propositions regarding $\gl$-expanders. We refer the readers to \cite{peled2013lipschitz} for omitted proofs.

\begin{prop}[Connectivity]\label{prop:conn}
Let $G$ be an $n$-vertex $d$-regular $\lambda$-expander. For every $A, B \subseteq V(G)$ satisfying $|A||B| > \left( \frac{\lambda n}{d} \right)^2$, there is an edge joining $A$ and $B$.
\end{prop}

\begin{proof} By \eqref{def.expander},
\[e(A,B) \geq \frac{d}{n}|A||B| - \lambda\sqrt{|A||B|}.\]
The right side is positive if $|A||B|>\left( \frac{\lambda n}{d} \right)^2.$
\end{proof}

\begin{prop}[Vertex-expansion]\label{prop.expan.nonbi}
Let $G$ be an $n$-vertex $d$-regular $\gl$-expander. Then for every $A \sub V(G)$,
\[ \frac{|N(A)|}{|A|} \ge \left(\frac{d}{\gl}\left(1-\frac{|N(A)|}{n}\right)\right)^2 .\]
\end{prop}

\begin{proof}
By \eqref{def.expander},
\[ d|A| = e(A,N(A)) \leq \frac{d}{n} |A| |N(A)| + \lambda\sqrt{|A||N(A)|} ,\]
from which the conclusion follows after dividing both sides by $|A|$ and simplifying.
\end{proof}

\begin{prop}[Volume growth, \protect{\cite[Proposition~2.5]{peled2013lipschitz}}]\label{prop.volume}
Let $G$ be an $n$-vertex $d$-regular $\lambda$-expander. Then for every $v \in V(G)$ and integer $t \geq 0$,
\[ |B(v,t)| \geq \min\left\{ \frac{n}{2} , \left( \frac{d}{2\lambda} \right)^{2t} \right\}. \]
\end{prop}

\begin{cor}[Diameter, \protect{\cite[Corollary 2.6]{peled2013lipschitz}}]\label{prop.diam}
Let $G$ be a $n$-vertex $d$-regular $\gl$-expander. If $\gl<d/2$, then the diameter of $G$ is at most $\log(n)/\log(\frac{d}{2\gl}).$    
\end{cor}

\subsection{Entropy}\label{sec:prelim:entropy}

We breifly recall some basic notions about entropy. For a discrete random variable $X$, the \emph{(Shannon) entropy} of $X$ is
\[ H(X) = \sum_x \pr(X=x) \log\frac{1}{\pr(X=x)} ,\]
where by convention $0\log\frac{1}{0} = 0$. For discrete random variables $X$ and $Y$, interpreting $X$ conditioned on $Y$ taking value $y$ as a random variable, we get
\[ H(X|Y=y) = \sum_x \pr(X=x|Y=y) \log\frac{1}{\pr(X=x|Y=y)} .\]
The \emph{conditional entropy} of $X$ given $Y$, denoted by $H(X|Y)$, is
\beq{conEntro}
H(X|Y) = \sum_y \pr(Y=y) H(X|Y=y) = \sum_y \pr(Y=y) \sum_x \pr(X=x|Y=y) \log\frac{1}{\pr(X=x|Y=y)} .
\enq
We write $H(X_1,\ldots,X_n)=H((X_1,\ldots,X_n))$ for jointly distributed random variables $X_1,\ldots,X_n$.

\begin{lem}\label{entropyprop} The following holds for discrete random variables $X$, $Y$, $Z$, and $X_1, \dots X_n$.
\begin{enumerate}
\item $H(X) \le \log |\mathrm{Image}(X)|$, where $\mathrm{Image}(X)$ is the set of values $X$ takes with positive probability, with equality iff $X$ is uniform from its image.
\item $H(X|Y) \le H(X)$
\item $H(X,Y)=H(X)+H(Y|X)$
\item $H(X_1\dots X_n|Y) \le H(X_1|Y) + \cdots + H(X_n|Y)$
\item if $Z$ is determined by $Y$, then
$H(X|Y) \le H(X|Z)$
\item if $Z$ is determined by $X$, then
$H(X,Z|Y) = H(X|Y)$.
\item $ H(X|Z) \leq H(X|Y) + H(Y|Z)$
\end{enumerate}
\end{lem}

\begin{proof} As the first six statements are fairly standard, we only prove (7):
$H(X|Y)+H(Y|Z) \geq H(X|Y,Z)+H(Y|Z) = H(X,Y,Z) - H(Y,Z) + H(Y,Z) - H(Z) = H(X,Y|Z) \geq H(X|Z)$.
\end{proof}

We also need the following version of \emph{Shearer's Lemma} \cite{chung1986some}.

\begin{lem}\label{lem:Sh}
Let $X=(X_1, \dots, X_N)$ be a random vector, and suppose $\alpha:2^{[N]}\rightarrow \mathbb R^+$
satisfies
\[ \sum_{A\ni i}\alpha_A =1 \quad \forall i\in [N].\]
Then for any partial order $\prec$ on $[N]$, we have
\[ H(X)\leq \sum_{A\subseteq [N]}\alpha_AH\big(X_A|(X_i:i \prec A)\big),\]
where $X_A=(X_i:i\in A)$ and $i\prec A$ means $i\prec a$ for all $a\in A$.
\end{lem}

\section{Graph containers for expanders}\label{sec:containers}

Sapozhenko's graph container method \cite{sapozhenko1987number} is a technique for bounding the number of subsets of vertices with a prescribed vertex-expansion. It plays a key role in numerous combinatorial problems (e.g.~\cite{balogh2024intersecting, balogh2022sharp, galvin2003homomorphisms, galvin2004phase, hamm2019erdHos, jenssen2023homomorphisms, jenssen2023approximately, kahn2022number, li2023number, park2022note}). Typically, the method is applied on graphs with weak vertex-expansion, meaning $|N(A)| \geq (1 + \ep) |A|$ for all appropriate $A$ and some small $\ep$. While our $\lambda$-expanders have stronger vertex-expansion (Proposition~\ref{prop.expan.nonbi}), a key ingredient of our proofs is the observation that Sapozhenko's graph container method is significantly more powerful for $\gl$-expanders than what is given by the vertex-expansion guarantees alone.

For a graph $G$, $v \in V(G)$, and $g, k \in \mathbb{N}$, define
\beq{def:container}
\cH_k(v,g):=\{X \sub V(G): X \mbox{ is $k$-linked in $G$},~v\in X,~\mbox{and }|N(X)|=g\} .
\enq
The reader may think of the case $k=4$, since this is what we use in Section~\ref{sec:main}. Sapozhenko's graph container method bounds the size of $\cH_k(v,g)$.

\begin{lem}\label{lem:container:expa}
There exists a universal constant $C$ such that the following holds. Let $G$ be an $n$-vertex $d$-regular $\gl$-expander. Then for any $v\in V(G)$ and $d\le g \leq n/2$,
\[ |\mathcal{H}_k(v,g)| \leq \exp\left( \frac{Cg}{d} \max\left\{ k\log(4\lambda/\sqrt{d})\log d,~ \frac{\lambda^2}{d} \right\} \right) .\]
\end{lem}

Compare the bound in Lemma~\ref{lem:container:expa} to the bound given by the combination of Proposition~\ref{prop.expan.nonbi} and Lemma~\ref{lem:numlinkset}: for $X \in \mathcal{H}_k(v,g)$, we have $|X| \leq \left(\frac{2\lambda}{d}\right)^2 g$, and hence
\[ |\mathcal{H}_k(v,g)| \leq (ed^k)^{(2\lambda/d)^2 g} \leq \exp\left( g \frac{4(k+1)\lambda^2}{d^2} \ln d \right) .\]
For $\lambda \gg \sqrt{d}$, Lemma~\ref{lem:container:expa} improves upon this bound.

The contribution of \Cref{lem:container:expa} over the original graph container lemma is two-fold. First, when $\gl \ll d$, it provides a better upper bound on $|\cH_k(v,g)|$ than the proof of the original version of graph container lemma, which yields an upper bound of $\exp_2\left[(1-\tilde{\Omega}(1))g\right]$ (see \cite[Lemma 3.1]{galvin2019independent} for a formal quantification; a slight improvement was later presented in \cite[Lemma 3.1]{park2022note}). Second, there is no requirement for $G$ to be bipartite.

Although our proof of \Cref{MT.nonbip} does not use \Cref{lem:container:expa} --- rather, we use the ingredients of the proof of \Cref{lem:container:expa} that we present in the rest of this section --- we think \Cref{lem:container:expa} may be of independent interest so we include its proof. Our proof of \Cref{lem:container:expa} follows the outline of the original proof by Sapozhenko, while providing better bounds by taking advantage of the nice properties of $\lambda$-expanders. It consists of two phases, with the first phase dividing into two subphases:
\begin{enumerate}[I.]
\item Construct containers/approximations.
\begin{enumerate}[(i)]
    \item Collect mutual covers of the members of $\cH_k(v,g)$ that serve as 'coarse' containers. (Lemma \ref{lem:contain1:expa}).
    \item Algorithmically refine these coarse containers into `$\psi$-approximations' (\Cref{lem:contain2:expa}).
\end{enumerate}
\item Reconstruct the elements of $\cH_k(v,g)$ from their $\psi$-approximations. 
\end{enumerate}

\begin{lem}\label{lem:contain1:expa}
There exists a universal constant $C_1 > 0$ such that the following holds. 
Let $G$ be an $n$-vertex $d$-regular $\lambda$-expander, and let $k \in \mathbb{Z}^+$. Then for any $v \in V(G)$ and for any $d \leq g \leq \frac{n}{2}$, there exists a family $\mathcal{V}_k(v,g) \subseteq 2^{V(G)}$ with 
\[ |\mathcal{V}_k(v,g)| \leq \exp\left( \frac{C_1 k \log(4\lambda/\sqrt{d})\log(d)}{d} g \right) \]
such that for every $X\in\cH_k(v,g)$, there is a mutual cover of $X$ in $\mathcal{V}_k(v,g)$.
\end{lem}

Note that Lemma~\ref{lem:contain1:expa} does not impose an upper bound on $\lambda$, and hence works for all graphs. For $\lambda$ on the order of $\sqrt{d}$, Lemma~\ref{lem:contain1:expa} is an improvement (saving a $\log(d)$ factor) over the usual ``first approximation'' used for Sapozhenko's graph container method. The sets of $\mathcal{V}_k(v,g)$ are `containers' in the sense that every $X \in \mathcal{H}_k(v,g)$ is contained in some $N(V)$ with $V \in \mathcal{V}_k(v,g)$.

\begin{proof}
Let $\ell = \left( \frac{4\lambda}{\sqrt{d}} \right)^4$. Given $X \in \mathcal{H}_k(v,g)$, let
\[ Q_0 = \{ u \in N(X) : d_X(u) \geq \ell \} ,\]
\[ X_0 = \{ u \in X : d_{Q_0}(u) \geq d/2 \} .\]
By Proposition~\ref{prop.expan.nonbi} and that $g \leq \frac{n}{2}$, we have that
\beq{eq.A1}
|X| \leq \left( \frac{2\lambda}{d} \right)^2 g .
\enq
This implies that
\beq{g0} |Q_0| \leq \frac{e(X,Q_0)}{\ell} \leq |X| \frac{d}{\ell} \leq \left( \frac{2\lambda}{d} \right)^2 \frac{d}{\ell} g = \frac{d}{64\lambda^2} g ,\enq
which is at most $n/32$ since $g \leq n/2$ and $\lambda \stackrel{\eqref{lambdalb}}{\geq} \sqrt{d}/2$. Applying~\eqref{def.expander}, we have that
\[ \frac{d}{2} |X_0| \leq e(X_0,Q_0) \leq \frac{d}{n} |X_0||Q_0| + \lambda\sqrt{|X_0||Q_0|} \stackrel{\eqref{g0}}{\leq} \frac{1}{32} d |X_0| + \lambda\sqrt{|X_0||Q_0|} ,\]
which after rearranging gives
\beq{eq.A0}
|X_0| \leq \left( \frac{4\lambda}{d}\right)^2 |Q_0| \stackrel{\eqref{g0}}{\leq} \frac{g}{4d} .
\enq

We produce a cover of $X$ of size at most $O\left( \frac{\log\ell}{d} g \right)$ as follows: choose a random subset $Y$ of $N(X) \setminus Q_0$, including each vertex independently with probability $\frac{\ln\ell}{d}$, and add to it a minimum size cover of $X \setminus N(Y)$. For $v \in X \setminus X_0$, we have that
\[ \pr\left( v \not\in N(Y) \right) \leq \left( 1 - \frac{\ln\ell}{d} \right)^{d/2} \leq \exp\left( - \frac{\ln\ell}{2} \right) = \frac{1}{\sqrt\ell}.\]
Thus the expected size of $|(X \setminus X_0) \setminus N(Y)|$ is at most $|X|/\sqrt\ell$, so by Markov's inequality, there exists a set $Y \sub N(X) \setminus Q_0$ of size at most $2g\ln\ell/d$ that covers all but at most $2|X|/\sqrt \ell$ vertices in $X \setminus X_0$. Observe that any $Z \sub X$ trivially has a cover in $N(X)$ of size at most $|Z|$. Therefore, there exists a cover of $X$ in $N(X)$ (that is, a mutual cover of $X$) of size at most
\[ \frac{2\ln\ell}{d} g + \frac{2}{\sqrt\ell} |X| + |X_0| \stackrel{\eqref{eq.A1},\eqref{eq.A0}}{\leq} \left( \frac{8\ln(4\lambda/\sqrt{d})}{d} + \frac{1}{2d} + \frac{1}{4d} \right) g \stackrel{\eqref{lambdalb1}}{\leq} \frac{7 \log(4\lambda/\sqrt{d})}{d} g .\]

By Proposition~\ref{prop:cover}, since $X$ is $k$-linked, any mutual cover of $X$ is $(k+2)$-linked. Let $\mathcal{V}_k(v,g)$ be the collection of all $(k+2)$-linked subsets of $V(G)$ of size at most $\frac{7 \log(4\lambda/\sqrt{d})}{d} g$ which contain at least one vertex from $B(v,1)$. By~\Cref{lem:numlinkset},
\[ |\mathcal{V}_k(v,g)|\leq (d+1) \cdot \sum_{m=1}^{\frac{7\log(4\lambda/\sqrt{d})}{d} g}(ed^{k+2})^{m-1} \leq \exp\left( \frac{7 (k+3) \log(4\lambda/\sqrt{d}) \ln(d)}{d} g \right) ,\]
and $\mathcal{V}_k(v,g)$ has a mutual cover of every element of $\mathcal{H}_k(v,g)$.
\end{proof}

Now we refine the coarse containers. Following \cite{galvin2019independent, jenssen2020independent}, for $1 \leq \psi \le d-1$, define a \emph{$\psi$-approximating pair} for $X \sub V(G)$ to be a pair $(S,F) \in 2^{V(G)} \times 2^{V(G)}$ satisfying
\beq{approx1} F \sub N(X),~S \supseteq X;\enq
\beq{approx2} d_{V(G) \setminus F} (u) \le \psi \quad \forall u \in S; \text{ and}\enq
\[ d_S(v) \le \psi \quad \forall v \in V \setminus F.\]

\begin{lem}\label{lem:contain2:expa}
There exists a universal constant $C_2$ such that the following holds. Let $G$ be a $d$-regular graph, and let $1\le \psi\leq d/2$. For each $V\subseteq V(G)$ there exists a family $\mathcal{W}(V, g)\subseteq 2^{V(G)}\times 2^{V(G)}$ with 
\[
|\mathcal{W}(V, g)| \leq \exp\left( \frac{C_2 \log(d)}{\psi} g \right)
\]
such that any $X$ mutually covered by $V$ with $|N(X)|=g$ has a $\psi$-approximating pair in $\mathcal{W}(V, g)$.
\end{lem}

The proof of \Cref{lem:contain2:expa} is very similar to the typical proof (see e.g. \cite[Lemma 5.5]{galvin2019independent}), except that we are not assuming $G$ is bipartite. This difference does not affect the proof in any way, so we will include the proof for completeness, but we will be brief.

\begin{proof}
Let $V \subseteq V(G)$ be given, let $X$ be mutually covered by $V$ with $|N(X)|=g$, and let $Q = N(X)$. Let $H \subseteq X$ be a minimum size set of vertices such that $F' := V \cup N(H) \sub N(X)$ satisfies for every $v \in X$, $d_{Q\setminus F'}(v) \leq \psi$. Let $S' = \{v \in V(G) : d_{F'}(v) \geq d-\psi\}$, and let $U \subseteq N(S') \setminus Q$ be a minimum size set of vertices such that $S = S' \setminus N(U)$ satisfies for every $v \not\in Q$, $d_S(v) \leq \psi$. Finally, let $F = F' \cup \{v \in V(G) : d_S(v) > \psi\}$.

We claim that $(S,F)$ is a $\psi$-approximating pair for $X$. Indeed, $F \subseteq N(X)$ since $F' \subseteq N(X)$ and if $d_S(v) > \psi$, then $v \in Q = N(X)$. Since $d_{F'}(v) = d - d_{Q\setminus F'}(v) \geq d-\psi$ for every $v \in X$, we have $X \subseteq S'$; since $U \cap Q = \emptyset$, we have $X \cap N(U) = \emptyset$, and hence $X \subseteq S' \setminus N(U) = S$. Note by definition that every $v \not\in F$ satisfies $d_S(v) \leq \psi$, and every $u \in S \subseteq S'$ satisfies $d_{V(G)\setminus F}(u) = d - d_F(u) \leq d - d_{F'}(u) \leq \psi$.

Now we must show that the number of choices for $H$ and $U$, and thus the number of choices for $(S,F)$, over all $X$ mutually covered by $V$ with $|N(X)|=g$ is at most $\exp\left( \frac{C_2 \log(d)}{\psi} g \right)$. We may select such an $H$ (not necessarily of minimum size) greedily, sequentially adding those $v \in X$ with $d_{Q \setminus F'}(v) > \psi$ to $H$, increasing the size of $F'$ by more than $\psi$ each time; hence $|H| \leq g/\psi$. Similarly, $|U| \leq |S'|/\psi$. Since $|S'|(d-\psi) \leq e(S',F') \leq |F'|d \leq gd$ and $\psi \leq d/2$, we have $|U| \leq \frac{d}{(d-\psi)\psi} g \leq \frac{2}{\psi} g$. Since $H \subseteq N(V)$ and $|N(V)| \leq |V|d \leq gd$, the number of choices for $H$ is at most $\binom{gd}{g/\psi} \leq \exp\left( \frac{3\ln(d)}{\psi} g \right)$; similarly, since $U \subseteq N(S')$ and $|N(S')| \leq |N(N(F'))| \leq gd^2$, the number of choices for $U$ is at most $\binom{gd^2}{2g/\psi} \leq \exp\left( \frac{8\ln(d)}{\psi} g \right)$, which finishes the proof.
\end{proof}

Combining Lemmas~\ref{lem:contain1:expa} and~\ref{lem:contain2:expa} immediately results in \Cref{lem:container:dre}.

\begin{lem}\label{lem:container:dre}
There exists a universal constant $C_3$ such that the following holds.
Let $G$ be an $n$-vertex $d$-regular $\lambda$-expander, and let $k \in \mathbb{Z}^+$. Then for any $v \in V(G)$, any $d \leq g \leq \frac{n}{2}$, and $1\le \psi \le d/2$, there is a family $\cU_k(v,g) \sub 2^{V(G)} \times 2^{V(G)} $ with
\[ |\cU_k(v,g)| \le \exp\left( C_3 \left( \frac{k \log(4\lambda/\sqrt{d})\log(d)}{d} + \frac{\log(d)}{\psi} \right) g \right)  \]
such that every $X \in \cH_k(v,g)$ has a $\psi$-approximating pair in $\cU_k(v,g)$.
\end{lem}

To prove Lemma~\ref{lem:container:expa}, we first choose a $\psi$-approximating pair $(S,F)$, and then bound the number of choices for $X$ for which $(S,F)$ is a $\psi$-approximating pair by simply using that $X \subseteq S$. To do this, an upper bound on $|S|$ is needed, which is nicely provided for expander graphs.

\begin{prop}\label{prop:SFbound}
Let $G$ be an $n$-vertex $d$-regular $\gl$-expander. 
Suppose that $(S,F)$ is a $\psi$-approximating pair for some $X\subseteq V(G)$. Then
\[ |S| \leq \left( \frac{\lambda}{d(1-|F|/n) - \psi} \right)^2 |F| .\]
\end{prop}

\begin{proof}
Observe that
\[(d-\psi)|S|\stackrel{\eqref{approx2}}{\le} e(S,F) \stackrel{\eqref{def.expander}}{\le} \frac{d}{n}|S||F|+\gl\sqrt{|S||F|} .\]
Subtracting $\frac{d}{n}|S||F|$ from both sides, dividing by $\sqrt{|S|}$, and simplifying finishes the proof.
\end{proof}

\begin{proof}[Proof of \Cref{lem:container:expa}]
Let $\psi=d/6$ and suppose that $(S, F)$ is a $\psi$-approximating pair of some $X\in \cH_k(v,g)$.
Note by~\eqref{approx1} that $|F|\le |N(X)|=g \leq n/2$, and therefore \Cref{prop:SFbound} gives
\[
|S| \leq \left( \frac{\lambda}{d(1-|F|/n) - \psi} \right)^2|F| \le \left( \frac{\lambda}{d/2 - \psi} \right)^2g \le \left( \frac{3\lambda}{d} \right)^2g.
\]
By applying \Cref{lem:container:dre} (with $\psi=d/6$) and~\eqref{approx1}, we obtain that
\[
|\mathcal{H}_k(v,g)| \le \sum_{(S, F)\in \cU_k(v,g)}2^{|S|} \le \exp\left( C\left( \frac{k\log(4\lambda/\sqrt{d})\log(d)}{d} + \frac{\lambda^2}{d^2} \right) g \right)
\]
for some absolute constant $C$.
\end{proof}

\section{Proof of Theorem~\ref{MT.nonbip}}\label{sec:main}

We begin this section by fixing our basic setup. After that, in Section~\ref{pf.overview}, we give a proof overview, highlighting the novelties compared to~\cite{peled2013lipschitz}; additionally, we fix a more advanced setup, allowing us to state the main technical lemma, Lemma~\ref{lem.nonbip}. In Section~\ref{sec:lem.nonbip}, we prove Lemma~\ref{lem.nonbip}, and hence Theorem~\ref{MT.nonbip}.

Recall that $G$ is a connected $n$-vertex $d$-regular $\gl$-expander, and $M\in \mathbb{Z}^+$ where $\lambda \leq d/5 \leq cn$ and~\eqref{eq.Mld} hold. Throughout this section, we fix an arbitrary vertex $w_0 \in V(G)$, and an integer $t\geq 2$.
By the translational symmetry of $\Z$, we may assume $k=0$: our goal is to give an upper bound on the probability that $f(w_0)\ge tM+2$, where $f\in_R \Lip_{0}^*(G;M)$.

For brevity, we write \[\cL := \Lip_0^*(G;M).\] 
For every $f \in \cL$, by the definition of $\cL$, we have $f(v) \in [0,M]$ for all but at most $(2\lambda/d) n$ vertices $v \in V(G)$.
We define the following:
\[\text{$A(f)$ is the 2-linked component of $\{v : f(v) \ge M+1\}$ containing $w_0$; and}\]
\[\text{$B(f)$ is the 4-linked component of $\{v : f(v) \ge 2M+2\}$ containing $w_0$.}\] 
Note that for $t\geq 2$, both $A(f)$ and $B(f)$ are well-defined and non-empty. We quickly observe that $B(f)$ is contained in the interior of $A(f)$.

\begin{prop}\label{obs.BinA}
For every $f \in \cL$, $B(f) \sub A(f)^\circ$, or equivalently, $B(f)^+ \sub A(f)$.
\end{prop}

\begin{proof}
Since $f$ is $M$-Lipschitz, for every $v$ with $f(v) \geq 2M+2$, if $u \in N(v)$, then $f(u) \geq 2M+2-M = M+2$. This implies that if $v$ and $v'$ are adjacent in $G^4$ and $f(v), f(v') \geq 2M+2$, then $v$ and $v'$ are in the same $2$-linked component of $\{v : f(v) \geq M+2\}$. Thus $B(f)$ is a subset of the $2$-linked component of $\{v : f(v) \geq M+2\}$ containing $w_0$, which is a subset of $A(f)$. Furthermore, for every $u \in B(f)^+$, we have $f(u) \geq M+1$, so $u \in A(f)$, that is, $B(f)^+ \subseteq A(f)$.
\end{proof}

In the proof of Proposition~\ref{obs.BinA}, it is required that the linked-ness of $B(f)$ is at most two more than the linked-ness of $A(f)$. We need the linked-ness of $B(f)$ to be at least $4$ for Proposition~\ref{prop:benchmark} to hold. While the choice of $M+1$ in the definition of $A(f)$ may be clear, since this is the condition for being outside of the `ground state' $[0,M]$, the choice of $2M+2$ in the definition of $B(f)$ is a bit more opaque. Informally, we a priori expect each vertex $v$ to have $M+1$ possibilities for $f(v)$, as is the case for $v$ which are in the ground state (in which case $f(v) \in [0,M]$). With these definitions, for $v \in A(f) \setminus B(f)$, we have $M+1$ choices for $f(v)$. 

\subsection{Proof overview}\label{pf.overview}

It is instructive to begin with the proof of Peled, Samotij, and Yehudayoff~\cite{peled2013lipschitz} (PSY) for their version of Theorem~\ref{MT.nonbip}. The \emph{flaws} of a Lipschitz function (and more generally, a graph homomorphism) are the vertices which are not in the `ground state,' as defined by Lemma~\ref{lem.Lips.GS}, that is, those vertices which do not take typical values. To bound the number of Lipschitz functions with a large fluctuation, we may employ the following general strategy: we first bound the number of such functions with a given set of flaws, and then we sum over all possible choices for these flaws, using say Lemma~\ref{lem:numlinkset} or Lemma~\ref{lem:container:expa}.

Via elementary combinatorial arguments, PSY compared the number of $f \in \Lip_{v_0}(G;M)$ with $A(f)$ equal to a given set $A$ to $|\Lip_{v_0}(G;M)|$, without computing either. In particular, they found that the probability that $f \in_R \Lip_{v_0}(G;M)$ has a flaw $A(f) = A$ is at most
\beq{eq.PSY} M(2|A|+1) (2M+1)^{|A|} \left( 1 - \frac{1}{M+1} \right)^{|\outbound(A)|} .\enq
In a vague sense, the `penalty' for having the flaw $A$ comes from the boundary of $A$, since there are fewer choices for $f(v)$ for $v \in \outbound(A)$ when $A$ is a flaw of $f$ than when it is not. While not exactly correct, here is an intuitive explanation of the $(1-\frac{1}{M+1})^{|\outbound(A)|}$ term in~\eqref{eq.PSY}: for an $f \in \Lip(G;M)$ with ground state $[0,M]$ (meaning most $f(v)$ are in $[0,M]$) there are typically $M+1$ choices for $f(v)$, while if $v$ is in the outer boundary of a flaw, then $f(v) \in [1,M]$ and there are only $M$ choices for $f(v)$, since $v$ is not in the flaw but is adjacent to a vertex $w$ with $f(w) \geq M+1$. The other terms in~\eqref{eq.PSY} come from the freedom the flaw allows $f$ by permitting \emph{any} value above $M$: if we have already determined $f(w)$ for some neighbor $w$ of $v$, then there are at most $2M+1$ choices for $f(v)$, namely $[f(w)-M,f(w)+M]$. (The $M(2|A|+1)$ term comes from that PSY used $\Lip_{v_0}(G;M)$ instead of $\Lip_0^*(G;M)$.) As long as $|\outbound(A)|$ is large enough compared to $|A|$ (which is true for good expanders by Proposition~\ref{prop.expan.nonbi}), these other terms are negligible when compared to the $(1-\frac{1}{M+1})^{|\outbound(A)|}$ penalty.

With~\eqref{eq.PSY} in hand, PSY summed over all possible choices of flaws using Lemma~\ref{lem:numlinkset}, which gives the number of $A$ of a particular size $a$ to be at most $(ed)^a$. As long as $a\log(d) \ll |\outbound(A)|/M$, the number of $A$ is negligible compared to~\eqref{eq.PSY} and the proof goes through. While Proposition~\ref{prop.expan.nonbi} suggests that $|A| \lesssim (\lambda/d)^2 |\outbound(A)|$, this is only true if $|A|$ is not too large. Since the flaws we must consider have size up to $\frac{2\lambda}{d} n$ by Lemma~\ref{lem.Lips.GS}, the expansion could be as poor as $|A| \lesssim (\lambda/d) |\outbound(A)|$, yielding the sufficient condition $(\lambda/d) \log(d) \ll 1/M$, or more precisely~\eqref{PSYcond}.

We improve the proof sketched above in three ways:

\begin{enumerate}[(a)]
\item \label{ei} Instead of considering flaws $A$ of size up to $\frac{2\lambda}{d} n$, we take advantage of the ordered structure of $\mathbb{Z}$ and consider \emph{2-fold flaws}, that is, pairs of sets $\{A(f),B(f)\}$, and apply our arguments to $B(f)$. Since $B(f)^+ \subseteq A(f)$, we have that $B(f)$ is much smaller than $A(f)$ and satisfies $|B(f)| \lesssim (\lambda/d)^2 |\outbound(B(f))|$ according to Proposition~\ref{prop.expan.nonbi}.

\item We compare the number of $f \in \Lip_0^*(G;M)$ with a given flaw using entropy techniques, in particular, Shearer's Lemma (\Cref{lem:Sh}). This essentially replaces the $(2M+1)^{|A|}$ factor in~\eqref{eq.PSY}, which accounted for the possible values of $f(v)$ with $v \in A$. Recall that the $(2M+1)^{|A|}$ factor comes from the $2M+1$ choices for $f(v)$ given $f(w)$ for some $w \in N(v)$. Instead of counting the number of possible $f_A$ vertex-by-vertex, where $f_A$ is $f$ restricted to $A$, Shearer's Lemma allows us to offset more possibilities for $f_{N(v)}$ with fewer possibilities for $f_v$. For example, if we know $a = \min f(N(v))$ and $b = \max f(N(v))$, then $b-M \leq f(v) \leq a+M$; thus the total number of possibilities for $f(w)$ and $f(v)$, where $w \in N(v)$, is at most $(b-a+1)((a+M)-(b-M)+1) \leq (M+1)^2$, which is better than the $(2M+1)^2$ bound given by the vertex-by-vertex method. The actual application of Shearer's Lemma and subsequent entropy calculations are more involved, but this example highlights the basic idea.

\item We replace the counting of the flaws using Lemma~\ref{lem:numlinkset} with a container approach combined with entropy. In fact, while using Lemma~\ref{lem:container:expa} (container method) would have already slightly improved the result of PSY, we make a crucial improvement by a delicate integration of the containers (via \Cref{lem:container:dre}) with our entropy approach. That is, instead of bounding the number of Lipschitz functions given a particular set of flaws, we bound the entropy of a random Lipschitz function given the containers for its flaws. Integrating entropy and graph containers has proven fruitful in some earlier graph homomorphism counting problems, for example,~\cite{jenssen2023homomorphisms, kahn2020number, li2023number}. Our approach seems closest to \cite{li2023number}, but we still require rather significant new ideas because of the different nature of the current problem.
\end{enumerate}

Another technical obstacle is that when counting Lipschitz functions under certain restrictions, one may need to consider not only $A(f)$ (and similarly for $B(f)$), the component of flaws containing $w_0$, but also other components of flaws that do not contain $w_0$. To get around this, several times throughout the proof we employ conditional probability, constraining the probability space of all Lipschitz functions on $V(G)$ to a smaller probability space of Lipschtiz functions on some $X \subseteq V(G)$, with boundary conditions outside of $X$ fixed. We do this because the event under study --- whether $f(w_0)>tM+2$ --- primarily depends on what occurs within $X$, rather than outside of $X$. 

Formally, we first partition the probability space $\cL ~(:=\Lip_0^*(G;M))$ according to $A(f)$ as follows. 
For any 2-linked set $A\subseteq V(G)$ with $w_0\in A$, and any function $p:  V\setminus A \rightarrow \mathbb{Z}$,
define the event
\[
\cL_p:=\{f \in \cL: A(f)=A, ~\text{and $f$ agrees with $p$ on $V\setminus A$}\} ,
\]
where by ``$f$ agrees with $p$ on $V \setminus A$'' we mean for all $v \in V \setminus A$, $f(v) = p(v)$. We suppress the dependence of $\cL_p$ on $A$ since, technically, the domain of $p$ tells us $A$; in what follows, 
\[\text{whenever we work with $p$, we will implicitly assume that $A$ is understood.}\]

Note that the events $\cL_p$ above are mutually exclusive, collectively exhaustive, and therefore only finitely many are non-empty. We can thus apply the law of total probability as follows: for $f\in_R \cL$,
\[
\pr(f(w_0)\ge tM+2) = \sum_{A, p} \pr(f(w_0)\ge tM+2 \mid f\in \cL_p)\cdot \pr(f\in \cL_p).
\]
\Cref{MT.nonbip} then follows immediately from the lemma below.

\begin{lem}\label{lem.nonbip} Suppose $G$ and $M$ satisfy the assumptions of \Cref{MT.nonbip}. Then for any 2-linked set $A\subseteq V(G)$ with $w_0\in A$, and any function $p: V(G) \setminus A \rightarrow \mathbb Z$, the following holds. If $f$ is chosen uniformly at random from $\mathcal L_p$, then
\beq{ML.ineq}\pr(f(w_0)\ge tM+2) \le \exp_2\left(-\frac{|B(w_0,t-1)|}{5M}\right). \enq    
\end{lem}

The rest of this section is devoted to the proof of Lemma~\ref{lem.nonbip}.

\subsection{Proof of Lemma~\ref{lem.nonbip}}\label{sec:lem.nonbip}
From now on, we work with $\cL_p$ for a fixed $A \sub V(G)$ and $p:V(G) \setminus A \ra \mathbb Z$. We further partition the probability space $\mathcal{L}_p$ using $B(f)$: for $g\in\mathbb{Z}^+$, let
\[\mathcal{B}(g) := \{ B(f) : f \in \cL_p, |N(B(f))| = g \},\]
and note that $\mathcal{B}(g)\subseteq \cH_4(w_0, g)$ (see the definition in~\eqref{def:container}).
By the law of total probability, we could write the left side of \eqref{ML.ineq} as
\[
\pr(f(w_0)\ge tM+2) = \sum_{g\in\mathbb{Z}^+}\sum_{B\in\mathcal{B}(g)} \pr(f(w_0)\ge tM+2,~B(f)=B),
\]
and then we could bound the probabilities on the right side in terms of $g$. However, as is typical in these kinds of problems, it is more effective to cover the events on right side by a much smaller number of events that we can still effectively bound. This is the purpose of the container method, \Cref{lem:container:dre}.

To this end, set
\beq{eq.psi}
\psi = \frac{d^{3/2}}{20\lambda} ,
\enq
and note that $1 \leq \psi \stackrel{\eqref{lambdalb1}}{\le} d/2$. Apply \Cref{lem:container:dre} with $k=4$ and $w_0$ playing the role of $v$ to obtain a family $\mathcal{U}(g) := \mathcal{U}_4(w_0, g) \subseteq 2^{V(G)} \times 2^{V(G)}$ satisfying
\beq{eq.containercost}
|\mathcal{U}(g)| \leq \exp\left(O\left(\frac{\log(4\lambda/\sqrt{d})\log d}{d} + \frac{\log d}{\psi}\right)g\right) \stackrel{\eqref{lambdalb1}, \eqref{eq.psi}}{=} \exp_2\left(O\left( \frac{\lambda\log d}{d^{3/2}} \right)g\right)
\enq
such that every $B \in \mathcal{B}(g)$ has a $\psi$-approximating pair $(S,F)$ in $\mathcal{U}(g)$.

Note that, by Proposition~\ref{obs.BinA}, we may assume that for every $(S,F) \in \mathcal{U}(g)$,
\beq{S+A}S \sub A^\circ, \text{ or equivalently, } S^+ \sub A ,\enq
since otherwise, we may simply replace each $S$ with $S \cap A^\circ$, because this replacement still satisfies the definition of $\psi$-approximating pair. We make a simple observation for future reference:
\beq{Obs.S+linked} \text{$S$ and $S^+$ are 4-linked.}
\enq
To see this for $S^+$, note that by the definition of $(S, F)$ (see~\eqref{approx1} and~\eqref{approx2}), for every $v \in S^+$,
\[\dist_G(v,B(f)) \leq \dist_G(v,F)+1 \leq \dist_G(v,S)+2 \leq 3.\]
Given that $B(f) \sub S^+$ and $B(f)$ is $4$-linked, we conclude that $S^+$ is also $4$-linked. 
Similarly, $S$ is $4$-linked.

Instead of partitioning the probability space $\mathcal{L}_p$ directly using $B(f)$, we cover $\cL_p$ using $\psi$-approximating pairs as follows:
\[ \pr(f(w_0)\ge tM+2) \le \sum_{g\in\mathbb{Z}^+}\sum_{(S, F)\in\mathcal{U}(g)} \pr(f(w_0)\ge tM+2,~\text{$(S, F)$ is a $\psi$-approx. pair of $B(f)$}).\]
Fix $g \in \mathbb Z^+$ and $(S,F) \in \mathcal{U}(g)$. 
We once again use conditional probability to establish another boundary condition and eliminate extraneous information. Fix $p':V(G) \setminus S^+ \to \mathbb{Z}$ that agrees with $p$ on $V(G) \setminus A$ (this is well-defined by \eqref{S+A}). Given $p$ and $p'$, we define the following two key events:
\[\cL_{p,p'}:=\{f \in \cL_p: \text{$f$ agrees with $p'$ on $V(G) \setminus S^+$}\};\]
\[\cL_{p, p'}(g,S,F):=\{f \in \cL_{p,p'}: \text{$B(f)\in \mathcal{B}(g)$, and $(S, F)$ is a $\psi$-approx. pair of $B(f)$}\}.\]
We aim to establish the following uniform upper bound: for any $p$, $g$, $S$, $F$, and $p'$ such that $\cL_{p,p'}$ and $\cL_{p, p'}(g,S,F)$ are nonempty, we have
\beq{pp'.bd0} \frac{|\cL_{p, p'}(g,S,F)|}{|\cL_{p,p'}|} \le \exp_2\left(-\frac{g}{3M}\right).\enq

\begin{proof}[Derivation of \Cref{lem.nonbip} from \eqref{pp'.bd0}]
Observe that for every $M$-Lipschitz $f\in \cL$ with $f(w_0) \ge tM+2$, it follows that $B(f)$ contains a ball of radius $t-2$ about $w_0$, that is, $B(w_0,t-2) \subseteq B(f)$,
so $|B(f)^+| \ge |B(w_0,t-1)|$. Note that, by applying \Cref{prop.expan.nonbi} to $N(B(f))$ (with the assumption that $\gl\le d/5$ and the fact that $|N(B(f))|\le |A(f)|\le \frac{2\gl}{d}n$), we have $|N(B(f))| \ge 3|B(f)|,$ so
\[|N(B(f))| \ge |B(f)^+|-|B(f)|\ge |B(f)^+|-|N(B(f))|/3,\]
and thus $|N(B(f))| \ge \frac{3}{4}|B(f)^+|  \ge \frac{3}{4}|B(w_0,t-1)|$.
Therefore, for $f \in_R \cL_p$, we have
\[
\begin{split}
\pr(f(w_0)&\ge tM+2) \le \sum_{g \ge \frac{3}{4} |B(w_0, t-1)|}\sum_{(S, F)\in\mathcal{U}(g)} \pr(\text{$(S, F)$ is a $\psi$-approx. pair of $B(f) \in \mathcal{B}(g)$})
\\
&= \sum_{g \ge \frac{3}{4} |B(w_0, t-1)|}\sum_{(S, F)\in\mathcal{U}(g)} \sum_{p'}\pr(\text{$(S, F)$ is a $\psi$-approx. pair of $B(f) \in \mathcal{B}(g)$} \mid f\in \cL_{p, p'})\cdot \pr(f\in \cL_{p, p'})
\\
&
=\sum_{g \ge \frac{3}{4} |B(w_0, t-1)|}\sum_{(S,F)\in \cU(g)}\sum_{p'}\frac{|\mathcal L_{p,p'}(g,S,F)|}{|\cL_{p,p'}|}\frac{|\cL_{p,p'}|}{|\cL_{p}|} \\
&\stackrel{\eqref{pp'.bd0}}{\le} 
\sum_{g \ge \frac{3}{4}|B(w_0, t-1)|}\sum_{(S,F)\in \cU(g)} \exp_2\left(-\frac{g}{3M}\right)\\
& \stackrel{\eqref{eq.containercost}}{\le} \sum_{g\geq \frac{3}{4}|B(w_0,t-1)|} \exp_2\left( O\left( \frac{\lambda\log d}{d^{3/2}} \right)g - \frac{g}{3M} \right)\\
&\stackrel{(\dagger)}{\le} \sum_{g \geq \frac{3}{4}|B(w_0,t-1)|} \exp_2\left( -\frac{g}{3.1M} \right) \leq \exp_2\left( - \frac{|B(w_0,t-1)|}{5M} \right),\end{split}\]
where $(\dagger)$ uses the assumption~\eqref{eq.Mld} that $M \le c \frac{d^{3/2}}{\lambda\log d}$, and we can choose $c$ so that $1/c$ dominates the universal constant in \eqref{eq.containercost}.
\end{proof}

\begin{proof}[Proof of \eqref{pp'.bd0}] 
We first establish a lower bound on $|\cL_{p,p'}|$, which will later serve as a `benchmark' for comparing our more involved upper bound on $|\cL_{p,p'}(g,S,F)|$. 
We will need the following simple observation for both bounds: for every $f \in \cL_{p,p'}$ and $v \in \inbound(S^+)$, 
\beq{fv.triv.obs} f(v) \leq M + \min_{w \in N(v) \setminus S^+} p'(w) .\enq

\begin{prop}\label{prop:benchmark} For $v \in S^+,$ set
\beq{def.uv} u_v:=\begin{cases}\min\left\{ 2M+1, M + \min_{w \in N(v) \setminus S^+} p'(w)\right\} & \text{ if } v \in \overline \partial(S^+)\\
2M+1 & \text{ if } v \in (S^+)^\circ.\end{cases}\enq
If $h: V(G) \to \Z$ agrees with $p'$ on $V(G) \setminus S^+$ and satisfies $h(v) \in [M+1,u_v]$ for $v \in S^+$, then $h \in \cL_{p,p'}$.
\end{prop}

\begin{proof}
Observe that $h$ satisfies  \eqref{def.phase} with $k=0$, because for every such function $h$, we have
\[
|v\in V(G)\setminus A: h(v)\in [0, M]| 
=|v\in V(G)\setminus A: p(v)\in [0, M]| \geq \left(1 - \frac{2\lambda}{d}\right)n\]
where the last inequality holds for any valid $p$ (that is, such that $\cL_{p, p'}\neq \emptyset$). 
Thus it suffices to verify that $h$ is an $M$-Lipschitz function, that is, $h$ satisfies $|h(v)-h(w)| \leq M$ for all $vw \in E(G)$. 

The cases where $ v, w \in S^+$ and $v, w \in V(G) \setminus S^+$ are immediate. For $v \in S^+$, $w \in V(G) \setminus S^+$, since $vw \in E(G)$, we must have $v \in \inbound(S^+)$ and $w \in \outbound(S^+)$. By the definition of $u_v$, we have that $h(v) \leq M+p'(w) = M+h(w)$. 
It remains to show that 
\beq{h.ub}p'(w) \leq 2M+1, ~\text{and thus }h(w)=p'(w)\leq M+h(v).\enq
This is indeed a property of (valid) $p'$, independent of the choice of $h$.
To prove it, take an arbitrary $f\in \cL_{p,p'}(g, S, F)$ (which we assume in non-empty), and note that $f(w)=p'(w)$.
Recall from~\eqref{approx1} that $F \subseteq N(B(f)) \subseteq N(S)$, and therefore $\dist_G(w,B(f)) \leq \dist_G(w,F)+1 \leq \dist_G(w,S)+2 \leq 4$. 
Then we must have $f(w)\leq 2M+1$; otherwise, $w$ would have been contained in $B(f)$, the 4-linked component of $\{v : f(v) \ge 2M+2\}$, which contradicts the fact that $w\notin S\supseteq B(f)$.
\end{proof}

Proposition~\ref{prop:benchmark} immediately yields the lower bound
\beq{eq.benchmark}
|\cL_{p,p'}| \geq \prod_{v \in S^+} (u_v-M) .
\enq

The most difficult and technical part of the proof is establishing an upper bound on $|\cL_{p, p'}(g, S, F)|$ using entropy.
Roughly speaking, we will apply Shearer's Lemma (\Cref{lem:Sh}) to a uniformly chosen $f \in \cL_{p,p'}(g,S,F)$ under a particular ordering of $S^+$ and then show that $|\cL_{p,p'}(g,S,F)|$ is significantly smaller than the `benchmark' in \eqref{eq.benchmark}.

\begin{prop}\label{prop.ordering}
There exists an ordering $<$ on $S^+$ that satisfies the following properties:
\begin{enumerate}[(i)]
\item the first vertex (that is, the smallest under $<$) is within distance $2$ from $V(G) \setminus S^+$;
\item every vertex of $S$ precedes every vertex of $\outbound(S)$ in $<$;
\item for every $v \in S^+$ besides the first under $<$, there exists a vertex preceding $v$ under $<$ at distance at most $4$ in $G$ from $v$.
\end{enumerate}
\end{prop}

\begin{proof}
    Pick a vertex $v_0$ in $S$ that is of distance 2 from $V(G) \setminus S^+$, which must exist because $V(G) \setminus S^+ \neq \emptyset$. Set $v_0$ to be the first vertex under $<$. Note that by \eqref{Obs.S+linked}, $G^4[S]$ is connected. We order vertices in $S$ in increasing order according to their distance from $u_0$ in $G^4[S]$, and then put all of $\outbound(S)$ after them, breaking ties arbitrarily, to create $<$.
    \end{proof}

Let $N^-_v := \{w \in N(v) : w < v\}$ and $d^-_v := |N^-_v|$.
For brevity, we write $d_{\overline{S^+}}(v):=d_{V(G)\setminus S^+}(v)$, that is, the number of edges from $v$ to $V(G) \setminus {S^+}$. 
For any $W \subseteq S^+$ and $v \in S^+$,
\[
\text{$v<W$ means for all $w \in W$, we have $v<w$}.
\]
For a function $f\in\cL_{p,p'}$, and a set $X\subseteq V(G)$,
\[\text{$f_X$ denotes the restriction of $f$ to $X$, and $f(X):=\{f(v): v \in X\}$.}\]
For a set $Y$ of integers, we use $\ext Y$ to denote $\{\min Y, \max Y\}$, that is, the minimum and maximum elements of $Y$. 

Observe that one copy of $N^-_v \cap S$, one copy of $N_v^-\cap \partial(S)$ and $d^-_v+d_{\overline{S^+}}(v)$ copies of $\{v\}$ for all $v \in S^+$ form a $d$-fold cover of $S^+$, in the sense that every vertex of $S^+$ is in precisely $d$ of these sets. 
Let $\bff$ be uniformly chosen from $\cL_{p, p'}(g,S,F)$. (We use the bold $\bff$ now to remind the reader that $\bff$ is random.)
By applying Shearer's Lemma (\Cref{lem:Sh}) with the ordering $<$ from \Cref{prop.ordering} and the above $d$-fold cover, we obtain that
\[\begin{split}
\log|&\cL_{p, p'}(g,S,F)| = H(\bff) = H(\bff_{S^+}) \\
\leq~& \frac{1}{d} \sum_{v \in S^+} \left[\left(d^-_v+d_{\overline{S^+}}(v)\right) H\left(\bff_v\big|\bff_{\{u: ~u<v\}}\right) + H\left(\bff_{N^-_v\cap S}\big|\bff_{\{u:~u<N^-_v\cap S\}}\right) + H\left(\bff_{N^-_v\cap \partial S}\big|\bff_{\{u:~u<N^-_v \cap \partial S\}}\right) \right]; \nonumber
\end{split}\]
we note that in the above expression, we interpret $\bff_\emptyset$ as the empty function, and hence $H(\bff_\emptyset) = 0$ and $H(X | \bff_\emptyset) = H(X)$; furthermore, $\{u:u<\emptyset\}=V(G)$, so $H(X | \bff_{\{u : u<\emptyset\}}) = H(X | \bff) = 0$ for any random variable $X = X(\bff)$.

The above expression is equal to 
\[\begin{split}
& \frac{1}{d} \sum_{v \in S^+} \left[ d^-_v H\left(\bff_v\big|\bff_{\{u:~u<v\}}\right) + H\left(\bff_{N^-_v \cap S}\big|\bff_{\{u:~u<N^-_v \cap S\}}\right) + H\left(\bff_{N^-_v \cap \partial S}\big|\bff_{\{u:~u<N^-_v \cap \partial S\}}\right) \right]\\
& \hspace{10cm}+ \sum_{v\in \outbound(S)} \frac{d_{\overline{S^+}}(v)}{d} H\left(\bff_v\big|\bff_{\{u:~u<v\}}\right)
\end{split}\]
\begin{align}
=&\frac{1}{d}\sum_{v \in S} \left[ d^-_v H\left(\bff_v\big|\bff_{\{u:~u<v\}}\right) + H\left(\bff_{N^-_v}\big|\bff_{\{u:~u<N^-_v\}}\right) \right] \label{entro:term1}\\
&\hspace{2cm}+ \frac{1}{d}\sum_{v \in \partial(S)} \left[ d^-_v H\left(\bff_v\big|\bff_{\{u:~u<v\}}\right) + H\left(\bff_{N^-_v \cap S}\big|\bff_{\{u:~u<N^-_v \cap S\}}\right) + H\left(\bff_{N^-_v \cap \partial S}\big|\bff_{\{u:~u<N^-_v \cap \partial S\}}\right)\right]\label{entro:term2}\\ 
&\hspace{10.5cm}+\frac{1}{d}\sum_{v\in \outbound(S)} {d_{\overline{S^+}}(v)} H\left(\bff_v\big|\bff_{\{u:~u<v\}}\right)\label{entro:term3} ,
\end{align}
where we use the facts that for every $v \in S$, $d_{\overline{S^+}}(v)=0$ and $N_v^- \subseteq S$, since every vertex of $S$ comes before every vertex of $\outbound(S)$ under $<$. 

First, applying Property (7) in \Cref{entropyprop} with $Y=\ext \bff(N^-_v)$, we obtain that~\eqref{entro:term1} is at most $1/d$ times
\[
\begin{split}
&~\sum_{v \in S} \left[ d^-_v H\left(\bff_v\big|\bff_{\{u:~u<v\}}\right) + H\left(\bff_{N^-_v}\big|\ext \bff(N^-_v)\right) \right] + \sum_{v \in S} H\left(\ext \bff(N^-_v) \big| \bff_{\{u:u<N^-_v\}}\right)\\
\le~&\sum_{v \in S}\sum_{w \in N^-_v}\left[H(\bff_v \big| \ext \bff(N_v^-)) + H(\bff_w \big| \ext \bff(N_v^-)) \right]+ \sum_{v \in S} H\left(\ext \bff(N^-_v) \big| \bff_{\{u:u<N^-_v\}}\right),
    \end{split}
    \]
where the inequality above follows from Properties (4) and (5) in \Cref{entropyprop}.

To derive an upper bound on~\eqref{entro:term2} and~\eqref{entro:term3}, we use the following key observation:
\[\text{ $B(f)$ is fully determined by $f_S$} .\]
The second line~\eqref{entro:term2}, again using Property (7) in \Cref{entropyprop} with $Y=\ext\bff(N^-_v)$ or $Y=(B(\bff), \ext \bff(N^-_v))$, is at most $1/d$ times
\[\begin{split} 
& \sum_{v \in \partial(S)} \left[ d^-_v H\left(\bff_v\big|\bff_{\{u:~u<v\}}\right) + H\left(\bff_{N^-_v \cap S}\big|\ext\bff(N^-_v)\right) + H\left(\bff_{N^-_v \cap \partial S}\big|B(\bff), \ext \bff(N^-_v)\right)\right]\\
&\hspace{2cm} + \sum_{v \in \partial(S)} H\left(\ext\bff(N^-_v) \big| \bff_{\{u:~u<N^-_v \cap S\}}\right) + \sum_{v \in \partial(S)} H\left(B(\bff), \ext \bff(N^-_v)\big|\bff_{\{u:~u<N^-_v \cap \partial S\}}\right)\\
\stackrel{(\dagger)}{\le}~& \sum_{v \in \outbound(S)} \sum_{w \in N^-_v \cap S} \left[ H\left(\bff_v\big|\ext \bff(N^-_v)\right) + H\left(\bff_w\big|\ext \bff(N^-_v)\right) \right]\\
& \hspace{2cm} 
+ \sum_{v \in \outbound(S)} \sum_{w \in N^-_v \cap \partial(S)} \left[H\left(\bff_v\big|B(\bff),\ext \bff(N^-_v)\right)+ H\left(\bff_w\big|B(\bff),\ext \bff(N^-_v)\right)\right]\\
& \hspace{4cm}  + \sum_{v \in \partial(S)} H\left(\ext \bff(N^-_v) \big| \bff_{\{u:u<N^-_v\cap S\}}\right) + \sum_{v \in \partial(S)} H\left(\ext \bff(N^-_v) \big| \bff_{\{u:u<N^-_v\cap \partial S\}}\right)\\
\le~& \sum_{v \in \outbound(S)} \sum_{w \in N^-_v \cap S} \left[ H\left(\bff_v\big|\ext \bff(N^-_v)\right) + H\left(\bff_w\big|\ext \bff(N^-_v)\right) \right]\\
& \hspace{2cm} 
+ \sum_{v \in \outbound(S)} \sum_{w \in N^-_v \cap \partial(S)} \left[H\left(\bff_v\big|B(\bff),\ext \bff(N^-_v)\right)+ H\left(\bff_w\big|B(\bff),\ext \bff(N^-_v)\right)\right]\\
& \hspace{10cm}  + 2 \sum_{v \in \partial(S)} H\left(\ext \bff(N^-_v) \big| \bff_{\{u:u<N^-_v\}}\right),\end{split}\]
where ($\dagger$) relies on the crucial fact that $B(\bff)$ is determined by $\bff_{\{u: u<\outbound(S)\}} = \bff_S$. This fact also tells us that the last term,~\eqref{entro:term3}, is at most $1/d$ times
\[\sum_{v\in \outbound(S)} {d_{\overline{S^+}}(v)} H\left(\bff_v\big|B(\bff)\right).\]
Combining the above inequalities, $\log|\cL_{p, p'}(g,S,F)|$ is at most $1/d$ times
\begin{multline}\label{eq.mainentropy}
\sum_{v \in S} \sum_{w \in N^-_v}  \left[ H\left(\bff_v \big|\ext \bff(N^-_v)\right) + H\left( \bff_w\big|\ext \bff(N^-_v)\right) \right]
+ \sum_{v \in \outbound(S)} \sum_{w \in N^-_v \cap S} \left[ H\left(\bff_v\big|\ext \bff(N^-_v)\right) + H\left(\bff_w\big|\ext \bff(N^-_v)\right) \right]\\
+ \sum_{v \in \outbound(S)} \sum_{w \in N^-_v \cap \partial(S)} \left[ H\left(\bff_v \big|B(\bff),\ext \bff(N^-_v)\right) + H\left(\bff_w \big|B(\bff),\ext \bff(N^-_v)\right)\right]
+ 2 \sum_{v \in S^+} H\left(\ext \bff(N^-_v) \big|\bff_{\{u:~u<N^-_v\}}\right) \\
+ \sum_{v\in \partial(S)} {d_{\overline{S^+}}(v)} H\left(\bff_v\big|B(\bff)\right).
\end{multline}
We bound each of the five sums in~\eqref{eq.mainentropy} individually. To give the reader some foresight, our aim is to bound~\eqref{eq.mainentropy} relative to~\eqref{eq.benchmark}: each of the $H(f_v | \cdot)$ and $H(f_w | \cdot)$ terms we aim to bound by $\log(u_v - M)$ and $\log(u_w - M)$, where $u_v$ is defined by~\eqref{def.uv}. Of course, if we only proved this bound, then we would prove exactly the benchmark~\eqref{eq.benchmark} (plus a bit from the fourth sum). The point is that we may obtain stronger bounds when conditioning on $B(f)$, as in the third and fifth sum. We tackle these five sums in order, as the first sum provides a good warm-up for the second sum, which provides a good warm-up for the third sum.

\medskip

\nin (i) \underline{First sum of~\eqref{eq.mainentropy}.} For an $M$-Lipschitz function $f$, we have that for every vertex $v$, $\max f(N^-_v)-M \leq f(v) \leq \min f(N^-_v)+M$. 
Hence, for any $v \in S$, $w \in N_v^-$, and $a \leq b$ such that $\ext\bff(N_v^-) = \{a,b\}$ with positive probability, we have
\[ H\left(\bff_v\big|\ext \bff(N^-_v)=\{a, b\}\right) \leq \log\left[ (a+M) - (b-M) + 1 \right]\]
and 
\[ H\left(\bff_w\big|\ext \bff(N^-_v)=\{a, b\}\right) \leq \log\left(b - a + 1 \right) .\]
Observe that for any integers $a \leq b$, we have that
\[ [(a+M)-(b-M)+1](b-a+1) = [2(M+1)-(b-a+1)](b-a+1) \leq (M+1)^2 .\]
Thus, 
\[H\left(\bff_v\big|\ext \bff(N^-_v)=\{a, b\}\right)+H\left(\bff_w\big|\ext \bff(N^-_v)=\{a, b\}\right) \le \log (M+1)^2.\]

Therefore, using the definition of conditional entropy (see~\eqref{conEntro}), the first sum of \eqref{eq.mainentropy} is at most
\beq{eq.sum1}\begin{split} &\sum_{v \in S} \sum_{w \in N^-_v} \sum_{a, b} \pr(\ext \bff(N_v^-)=\{a, b\})\left[ H\left(\bff_v\big|\ext \bff(N^-_v)=\{a, b\}\right)+H\left(\bff_w\big|\ext \bff(N^-_v)=\{a, b\}\right) \right]\\ &\leq~  \sum_{v \in S} \sum_{w \in N^-_v} \log\left(M+1\right)^2 \stackrel{(\star)}{=}  \sum_{v \in S} \sum_{w \in N^-_v} \log\left[(u_v-M)(u_w-M) \right] ,
\end{split}\enq
where in $(\star)$ we recall that for $v \in S \subseteq (S^+)^{\circ}$, $u_v = 2M+1$ (see \eqref{def.uv}), and any $w\in N^-_v$ remains in $S$ by our choice of the ordering $<$ (see Proposition~\ref{prop.ordering}).

\medskip

\nin (ii) \underline{Second sum of~\eqref{eq.mainentropy}.} For ease of notation, for $v \in \partial(S), w \in N^-_v\cap S$, and $a \le b$, let
\[ T(v,w;a,b) := H\left(\bff_v\big|\ext \bff(N^-_v)=\{a,b\}\right) + H\left(\bff_w\big|\ext \bff(N^-_v)=\{a,b\}\right). \]

\begin{prop}\label{obs.f.ub.uv}
    Let $f \in \cL_{p,p'}(g,S,F)$. For every $v \in \partial (S)$, 
    \beq{obs.fv.uv}f(v) \le u_v.\enq
\end{prop}

\begin{proof}
Recall the definition of $u_v$ in \eqref{def.uv}.
For all $v \in \partial(S)$, it follows identically from the proof of \eqref{h.ub} that $f(v) \le 2M+1$.
Then for $v \in \partial(S) \cap (S^+)^\circ$, \eqref{obs.fv.uv} immedialty follows from the defintion of $u_v$. For $v \in \partial(S) \cap \overline\partial(S^+),$ \eqref{obs.fv.uv}  follows from \eqref{fv.triv.obs}.
\end{proof}

\begin{lem}\label{prop.T.bd0} For any $v \in \partial(S)$, $w \in N^-_v \cap S$, and valid $a\le b$ (meaning $\ext\bff(N_v^-) = \{a,b\}$ with positive probability), we have
\[ T(v,w;a,b) \le \log \left[(u_v-M)(u_w-M)\right].\]
\end{lem}
\begin{proof}
For $f \in \cL_{p,p'}(g,S,F)$ with $\ext f(N_v^- \cap S) = \{a,b\}$, we have $M+1 \leq f(w) \leq b$ and $\max\{b-M,M+1\} \leq f(v) \leq u_v$, where the $M+1$ lower bounds follow from $S^+ \subseteq A$ (see~\eqref{S+A}). If $b \leq 2M+1$, this tells us
\[ T(v,w;a,b) \le \log(u_v-M)+\log(b-M) \stackrel{(\star)}{\le} \log\left[(u_v-M)(u_w-M)\right] ,\]
where in $(\star)$ we use that $u_w=2M+1$ since $w \in S \sub (S^+)^\circ$ (see definition~\eqref{def.uv}). If $b \geq 2M+1$, we have
\[\begin{split}
T(v,w;a,b)& \le \log(u_v-(b-M)+1)+\log(b-M)\\ 
&= \log\left[ (u_v-M)(M+1) + (u_v-b)(b-2M-1) \right] \\ 
&\stackrel{(\star)}{\le} \log\left[(u_v-M)(u_w-M)\right],\end{split}\]
where in $(\star)$ we use that $u_w=2M+1$ and $(u_v-b)(b-2M-1) \le 0$ since $u_v \leq 2M+1$ (see~\eqref{def.uv}) and $b \geq 2M+1$.
\end{proof}

By \Cref{prop.T.bd0}, the second sum of \eqref{eq.mainentropy} is 
\beq{eq.sum2.0} 
\sum_{v \in \outbound(S)} \sum_{w \in N^-_v \cap S} \sum_{a\le b} T(v,w;a,b)\cdot \pr(\ext \bff(N_v^-)=\{a,b\}) \le  \sum_{v \in \partial(S)} \sum_{w \in N_v^- \cap S}\log[(u_v-M)(u_w-M)].
\enq

\medskip

\nin (iii) \underline{Third sum of~\eqref{eq.mainentropy}.} For $v \in \partial(S)$, $w \in N_v^-\cap \partial(S)$, $a \le b$ and a set $B\in \mathcal{B}(g)$, we define
\[ T(v,w;a,b, B) := H\left(\bff_v\big|B(\bff)=B,\ext \bff(N^-_v)=\{a,b\}\right)+ H\left(\bff_w\big| B(\bff)=B, \ext \bff(N^-_v)=\{a,b\}\right).\]

Note that given $B(f)=B$, we can determine whether $v \in N(B(f)).$
Let $\mathbf{1}_{v \in N(B(f))}$ be the indicator function for whether $v \in N(B(f))$. 
\begin{prop}\label{prop.fv.case2.bd}  Let $f \in \cL_{p,p'}(g,S,F)$. For any $v \in \partial(S)$, given $\ext f(N^-_v)=\{a,b\}$ with $a \leq b$, we have
\beq{fv.case2.bd} \max\{b-M, M+1+\mathbf{1}_{v \in N(B(f))}\} \le f(v) \le \min\{a+M,u_v\}.\enq
\end{prop}

\begin{proof}
The upper bound is immediate from \Cref{obs.f.ub.uv}.
    For the lower bound, $f(v) \ge M+1$ because $\partial (S) \sub A$ (see \eqref{S+A}). Furthermore, if $v \in N(B(f))$, then $f(v) \ge (2M+2)-M=M+2.$ 
\end{proof}    

\begin{lem}\label{prop.T.bd} For any $v \in \partial(S)$, $w \in N^-_v \cap \partial(S)$, $a \le b$ and $B\in \mathcal{B}(g)$ such that $\ext\bff(N_v^-) = \{a,b\}$ and $B(\bff) = B$ with positive probability, we have
\[ T(v,w;a,b,B) \le \log \left[(u_v-M)(u_w-M)\right]-\frac{\mathbf{1}_{v \in N(B)}+\mathbf{1}_{w \in N(B)}}{M}.\]
\end{lem}

\begin{proof} 
As both $v,w$ are in $\partial(S)$, we apply the bound from \eqref{fv.case2.bd} to obtain \[H\left(\bff_v \big|B(\bff)=B, \ext \bff(N^-_v)=\{a,b\}\right)\le \log(u_v-M-\mathbf{1}_{v \in N(B)}),\]
and
\[H\left(\bff_w\big|B(\bff)=B, \ext \bff(N^-_v)=\{a,b\}\right)\le \log(u_w-M-\mathbf{1}_{w \in N(B)}).\]
Similar to \Cref{prop.T.bd0}, we have
 \[\begin{split} T(v,w;a,b, B) 
&
\le \left[\log (u_v-M-\mathbf{1}_{v\in N(B)}) +\log(u_w-M-\mathbf{1}_{w\in N(B)}) \right]\\
&= \log\left[ (u_v-M)(u_w-M) \right] + \left[\log\left(1-\frac{\mathbf{1}_{v\in N(B)}}{u_v-M}\right)\left(1-\frac{\mathbf{1}_{w\in N(B)}}{u_w-M}\right)\right] \\
&\stackrel{(\star)}{\leq} \log\left[ (u_v-M)(u_w-M) \right] + \left[\mathbf{1}_{v\in N(B)} \log\left(1-\frac{1}{M+1}\right)+\mathbf{1}_{w\in N(B)} \log\left(1-\frac{1}{M+1}\right)\right]\\
& \leq \log\left[ (u_v-M)(u_w-M) \right] - \frac{\mathbf{1}_{v\in N(B)}+\mathbf{1}_{w\in N(B)}}{M}, \end{split}\]
where in $(\star)$ we use that $u_v \leq 2M+1$ for all vertices in $S^+$, and in the final inequality we used the elementary inequality $\log_2(1-\frac{1}{M+1}) \leq -1/M$ for $M \geq 1$ (we do not optimize the constant factors).\end{proof}

By \Cref{prop.T.bd} and again the definition of conditional entropy (see~\eqref{conEntro}), the third sum of \eqref{eq.mainentropy} is 
\beq{eq.sum2} 
\begin{split} 
&\sum_{v \in \outbound(S)} \sum_{w \in N^-_v \cap \partial S} \sum_{a \leq b} \sum_B T(v,w;a,b, B)\cdot \pr(B(\bff)=B, \ext \bff(N_v^-)=\{a,b\})\\
\le & \sum_{v \in \partial(S)} \sum_{w \in N_v^- \cap \partial S}\log[(u_v-M)(u_w-M)]\\
&\hspace{2cm}-\sum_{v \in \partial(S)}\sum_{w \in N_v^-\cap \partial(S)} \sum_B \frac{\mathbf{1}_{v \in N(B)}+\mathbf{1}_{w \in N(B)}}{M} \sum_{a \leq b} \pr(B(\bff)=B, \ext \bff(N_v^-)=\{a,b\})\\
= & \sum_{v \in \partial(S)} \sum_{w \in N_v^-\cap \partial S}\log[(u_v-M)(u_w-M)]-\sum_{B} \pr(B(\bff)=B)\frac{e(N(B) \cap \partial(S), \partial(S))}{M},\\
\end{split}
\enq
where the last equality holds because the edges $vw$ with both endpoints in $N(B) \cap \outbound(S)$ are counted twice by both $\mathbf{1}_{v \in N(B)}+\mathbf{1}_{w \in N(B)}$ and $e(N(B) \cap \partial(S), \partial(S))$, while the edges $vw$ with only one endpoint in $N(B) \cap \outbound(S)$ are counted just once by both $\mathbf{1}_{v \in N(B)}+\mathbf{1}_{w \in N(B)}$ and $e(N(B) \cap \partial(S), \partial(S))$.

\medskip

\nin (iv) \underline{Fourth sum of~\eqref{eq.mainentropy}.} Let $v \in S^+$. By the definition of $<$ (see \Cref{prop.ordering}), if $N^-_v$ does not contain the first vertex of $<$,  there exists $w<N^-_v$ which is at most distance $4$ from some $u \in N^-_v$, and therefore $w$ is at most distance $6$ from every vertex of $N^-_v$. Thus for any $f\in \cL_{p, p'}(g,S,F)$,
\[ f(w)-6M \leq \min f(N^-_v) \leq \max f(N^-_v) \leq f(w)+6M ,\]
and hence
\[ H\left(\ext \bff(N^-_v) \big|\bff_{\{u:~u<N^-_v\}}\right) \leq 2\log\left( 12M+1 \right) .\]
If $N^-_v$ contains the first vertex of $<$, which by \Cref{prop.ordering} is at most distance 2 from some $w\in V(G) \setminus S^+$, then this $w$ is at most distance $4$ from every vertex of $N^-_v$. 
Note that $\bff(w) = p'(w)$ is fixed. We may apply the same argument as the previous case and conclude that
\[ H\left(\ext \bff(N^-_v) \big|\bff_{\{u:~u<N^-_v\}}\right) \leq 2\log\left( 8M+1 \right) .\]
In conclusion, the third sum of \eqref{eq.mainentropy} is
\beq{eq.sum3}
 2 \sum_{v \in S^+} H\left(\ext \bff(N^-_v) \big|\bff_{\{u:~u<N^-_v\}}\right) \leq {|S^+|} \cdot 4\log(12M+1).
\enq

\medskip

\nin (v) \underline{Fifth sum of~\eqref{eq.mainentropy}.} Recall from \Cref{obs.f.ub.uv}, \Cref{prop.fv.case2.bd}, and $S^+\subseteq A$ that for $v \in \outbound(S)$, we have $M+1+\mathbf{1}_{v \in N(B(f))} \leq f(v) \leq u_v$ for every $f \in \cL_{p,p'}(g,S,F)$. Thus the fourth sum of \eqref{eq.mainentropy} is
\beq{eq.sum4} 
\begin{split}
\sum_{v \in \outbound(S)} {d_{\overline{S^+}}(v)} H\left(\bff_v\big|B(\bff)\right)&=\sum_{v \in \outbound(S)} {d_{\overline{S^+}}(v)} \sum_{B} \mathbb{P}(B(\bff)=B) H\left(\bff_v\big|B(\bff)=B\right)
\\
&\leq \sum_{v \in \outbound(S)} {d_{\overline{S^+}}(v)} \log(u_v-M) + \sum_{B} \mathbb{P}(B(\bff)=B) \sum_{v\in\outbound(S)\cap N(B)}{d_{\overline{S^+}}(v)}\log\left(1 - \frac{1}{u_v-M}\right)\\
&\leq \sum_{v \in \outbound(S)} {d_{\overline{S^+}}(v)} \log(u_v-M) - \sum_{B} \mathbb{P}(B(\bff)=B) \frac{e(N(B)\cap\partial(S),\overline{S^+})}{M},
\end{split}
\enq
where in the last inequality we use that $u_v \leq 2M+1$ for all $v$ and $\log_2(1-\frac{1}{M+1}) \leq - 1/M$.

\medskip

Finally, putting \eqref{eq.sum1}, \eqref{eq.sum2.0} \eqref{eq.sum2}, \eqref{eq.sum3}, and \eqref{eq.sum4} into \eqref{eq.mainentropy}, we have that
\[\begin{split}
\log |\cL_{p,p'}(g,S,F)|& \leq \frac{1}{d}\left[\sum_{v\in S^+} \sum_{w \in N^-_v} \log\left[(u_v-M)(u_w-M) \right]
- \sum_B \mathbb{P}(B(\bff)=B) \frac{e(N(B)\cap\outbound(S),\outbound(S))}{M}\right.\\
&\quad\left.+ {|S^+|} \cdot 4\log(12M+1)+ \sum_{v \in \outbound(S)} {d_{\overline{S^+}}(v)} \log(u_v-M) - \sum_{B} \mathbb{P}(B(\bff)=B) \frac{e(N(B)\cap\outbound(S),\overline{S^+})}{M}\right] \\
& \stackrel{(\star)}{\le} \sum_{v \in S^+} \log(u_v-M) + \frac{1}{d} {|S^+|} \cdot 4\log(12M+1) - \frac{1}{d} \sum_B \mathbb{P}(B(\bff)=B) \frac{e(N(B)\cap\outbound(S),\overline{S})}{M},\end{split}\]
where $(\star)$ uses
\[\begin{split} &\sum_{v\in S^+} \sum_{w \in N^-_v} \log\left[(u_v-M)(u_w-M) \right]+\sum_{v \in \outbound(S)} {d_{\overline{S^+}}(v)} \log(u_v-M)\\
&=\sum_{v \in S^+}d_v^-\log(u_v-M)+\sum_{v \in S^+}\sum_{w \in N^-_v}\log(u_w-M)+ \sum_{v \in S^+} {d_{\overline{S^+}}(v)} \log(u_v-M)=d\sum_{v \in S^+} \log (u_v-M).\end{split}\]
Recall that $g = |N(B)| \leq \frac{2\lambda}{d}n$ and $\psi \leq d/10$ (see~\eqref{lambdalb1} and \eqref{eq.psi}), so by Proposition~\ref{prop:SFbound}, for $\lambda \leq d/5$, we have that
\beq{s.ub1} |S| \leq \left( \frac{\lambda}{d(1-g/n)-\psi} \right)^2 g \leq \left( \frac{\lambda}{d(1-2\lambda/d)-\psi} \right)^2 g \leq \left( \frac{2\lambda}{d} \right)^2 g \leq g/4 .\enq
Thus
\[ e(N(B)\cap\partial(S),\overline{S}) = d|N(B)\cap\partial(S)| - e(N(B)\cap\partial(S),S) \geq d|N(B)\cap\partial(S)| - d|S| \geq d(g-2|S|) \geq dg/2.\]
Also, we have
\[ |S^+|\le |S| + |F| + |N(S)\setminus F| \stackrel{\eqref{approx1},\eqref{approx2}}{\leq} g+(1+\psi)|S| \stackrel{\eqref{s.ub1}}{\leq} (1+(2\lambda/d)^2(1+\psi)) g \stackrel{\eqref{eq.psi}}{\le} (2+\lambda/(5d^{1/2})) g \stackrel{\eqref{lambdalb1}}{=} O(\lambda g/d^{1/2}) .\]
Summarizing the above, we obtain that
\[
\begin{split}
|\cL_{p,p'}(g,S,F)|
&\leq \prod_{v \in S^+} (u_v-M)\cdot\exp_2\left( O\left(\frac{\lambda g}{d^{3/2}}\right)\cdot4\log(12M+1) - \frac{g}{2M} \right) \\
&\leq \prod_{v \in S^+} (u_v-M)\cdot\exp_2\left( O\left(\frac{\lambda \log(d)}{d^{3/2}} g \right) - \frac{g}{2M} \right) \\
& \leq \prod_{v \in S^+} (u_v-M)\cdot \exp_2\left( -\frac{g}{3M} \right),
\end{split}\]
where the last inequality uses the assumption~\eqref{eq.Mld} that $M \le c \frac{d^{3/2}}{\lambda\log d}$, and we can choose $c$ so that $1/c$ dominates the implicit (universal) constant in $O\left(\frac{\gl \log (d)}{d^{3/2}}g\right)$.
This, together with the `benchmark'~\eqref{eq.benchmark}, leads to 
\[ \frac{|\cL_{p,p'}(g,S,F)|}{|\cL_{p,p'}|} \leq \exp_2\left( -\frac{g}{3M}\right),\]
which completes the proof of~\eqref{pp'.bd0}.
\end{proof}

\section{Proof of Theorem~\ref{thm.flat}}\label{sec.GS}
In this section, we derive \Cref{thm.flat} from \Cref{MT.nonbip}.
The core idea is to show that $\Lip^*_k(G;M)$ can be covered by at most $M+1$ translations of $\Lip_{v_0}(G;M)$.
This crucial proposition, combined with Bayes's rule, allows us to transfer our results from the fixed ground state model $\Lip^*_k(G;M)$ to the one-point boundary condition model $\Lip_{v_0}(G;M)$.

\begin{prop}\label{prop:gstobc}
For any $M \in \mathbb{Z}^+$, any $n$-vertex $d$-regular $\lambda$-expander $G$ with $\lambda \leq d/5$, and any $v_0 \in V(G)$, we have that for every $k \in \mathbb{Z}$,
\beq{Bayes} |\Lip^*_k(G;M)| \leq (M+1) |\Lip_{v_0}(G;M)| .\enq
\end{prop}

\begin{proof}
By symmetry, it suffices to prove it for $\Lip^*_0(G;M)$.
Recall that we assume $G$ is always connected; therefore, very loosely,
\beq{f.number.n} \text{if $f \in \Lip^*_k(G;M)$, then $k-nM \leq f(v_0) \leq k+M+nM$.}\enq
By Lemma~\ref{lem.Lips.GS}, $\Lip(G;M) = \bigcup_{k \in \mathbb{Z}} \Lip^*_k(G;M)$. For $f \in \Lip(G;M)$, let $\kappa(f)$ be the minimum $k$ such that $f \in \Lip^*_k(G;M)$. Observe that $\Lip^*_k(G;M)$ and $\Lip^*_{k'}(G;M)$ are disjoint if $|k-k'| > M$, since $\lambda < d/4$. Therefore given $f \in \Lip(G;M)$ and $\kappa:=\kappa(f)$,
\beq{obs.kappa}\{k:f \in \Lip^*_k(G;M)\} \sub [\kappa,\kappa+M]. \enq
Let $L \in \mathbb{Z}^+$. 
Using these observations and the translational symmetry of $\Lip(G;M)$ (i.e., for every $k$, there is a bijection between $\Lip^*_0(G;M)$ and $\Lip^*_k(G;M)$ defined by $f \mapsto f + k$), we have that
\begin{align*}
|\Lip^*_0(G;M)|
&= \frac{1}{L} \sum_{k=1}^L |\Lip^*_k(G;M)|
\stackrel{\eqref{obs.kappa}}{=} \frac{1}{L} \sum_{k=1}^L \sum_{k' = k-M}^k \left|\left\{ f \in \Lip^*_k(G;M) : \kappa(f) = k' \right\}\right| \\
&\leq \frac{1}{L} \sum_{k' = 1-M}^L \sum_{k=k'}^{k'+M} \left|\left\{ f \in \Lip^*_k(G;M) : \kappa(f) = k' \right\}\right| \\
&\leq \frac{1}{L} \sum_{k'=1-M}^L (M+1) \left|\left\{ f \in \Lip(G;M) : \kappa(f) = k' \right\}\right| \\
&\stackrel{\eqref{f.number.n}}{\leq} \frac{M+1}{L} \sum_{k'=1-M}^L \sum_{r=k'-nM}^{k'+M+nM} \left|\left\{ f \in \Lip(G;M) : \kappa(f) = k', f(v_0) = r \right\}\right| \\
&\leq \frac{M+1}{L} \sum_{r=1-M-nM}^{L+M+nM} \sum_{k'=r-M-nM}^{r+nM} \left|\left\{ f \in \Lip(G;M) : \kappa(f) = k', f(v_0) = r \right\}\right| \\
&\le \frac{M+1}{L} \sum_{r=1-M-nM}^{L+M+nM} \left|\left\{ f \in \Lip(G;M) : f(v_0) = r \right\}\right| \\
&= (M+1) \frac{L+2M+2nM}{L} |\Lip_{v_0}(G;M)| .
\end{align*}
Taking $L \to \infty$ finishes the proof.
\end{proof}

\begin{proof}[Proof of \Cref{thm.flat}]
Let $\bff$ be chosen from $\Lip_{v_0}(G;M)$ uniformly at random, and let $2M+2 \le L \le nM$ be an integer. 
Note that, since $G$ is connected,
\[\text{if $f \in \Lip_{v_0}(G;M)$, then $f \in \Lip^*_k(G;M)$ for some $-nM\leq  k \leq nM$.}\]
Then for any $v \in V(G)$,
\[\begin{split} \pr\left( \bff(v) \geq 2L \right) &=\frac{|\{f \in \Lip_{v_0}(G;M): f(v) \ge 2L\}|}{|\Lip_{v_0}(G;M)|}\leq \sum_{k=-nM}^{nM} \frac{\left|\left\{ f \in \Lip_{v_0}(G;M) : f \in \Lip^*_k(G;M), f(v) \geq 2L \right\}\right|}{|\Lip_{v_0}(G;M)|} \\
&\stackrel{\eqref{Bayes}}{\leq} (M+1) \sum_{k=-nM}^{nM} \frac{\left|\left\{ f \in \Lip^*_k(G;M) : f(v_0) = 0, f(v) \geq 2L \right\}\right|}{|\Lip^*_k(G;M)|}\\
&\leq (M+1) \left[\sum_{k=-nM}^L \frac{\left|\left\{ f \in \Lip^*_k(G;M) : f(v) \geq 2L \right\}\right|}{|\Lip^*_k(G;M)|} + \sum_{k=L}^{nM} \frac{\left|\left\{ f \in \Lip^*_k(G;M) : f(v_0) = 0 \right\}\right|}{|\Lip^*_k(G;M)|}\right] . \end{split}\]
Note that, by \Cref{MT.nonbip} and the translational symmetry of $\Lip(G;M)$,
\[\begin{split} 
\sum_{k=-nM}^L &\frac{\left|\left\{ f \in \Lip^*_k(G;M) : f(v) \geq 2L \right\}\right|}{|\Lip^*_k(G;M)|}
= \sum_{k=-nM}^L\frac{\left|\left\{ f \in \Lip^*_0(G;M) : f(v) \geq 2L-k \right\}\right|}{|\Lip^*_0(G;M)|} \\
&= \sum_{\ell = L}^{2L+nM} \frac{\left|\left\{ f \in \Lip^*_0(G;M) : f(v) \geq \ell \right\}\right|}{|\Lip^*_0(G;M)|} \le \sum_{\ell = L}^{2L+nM} \exp_2\left(-\frac{|B(v,\lfloor (\ell-2)/M\rfloor-1)|}{5M}\right).
\end{split}\]
Similarly, using the symmetry of $\Lip(G;M)$ that for every $k$, there is a bijection between $\Lip^*_k(G;M)$ and $\Lip^*_0(G;M)$ defined by $f \mapsto -f + k+M$,
\[\begin{split}\sum_{k=L}^{nM} &\frac{\left|\left\{ f \in \Lip^*_k(G;M) : f(v_0) = 0 \right\}\right|}{|\Lip^*_k(G;M)|} = \sum_{k=L}^{nM}\frac{\left|\left\{ f \in \Lip^*_0(G;M) : f(v_0) =k+M \right\}\right|}{|\Lip^*_0(G;M)|} \\
&\le \frac{\left|\left\{ f \in \Lip^*_0(G;M) : f(v_0) \geq L \right\}\right|}{|\Lip^*_0(G;M)|} \le \exp_2\left(-\frac{|B(v,\lfloor (L-2)/M\rfloor-1)|}{5M}\right).
\end{split}\]
To prove the bound on $R(\bff)$, set $L = M\ceil{ C'\left(\log\log n/\log(d/2\lambda)\right)+1} + 2$. By combining the above bounds and applying \Cref{prop.volume}, we obtain
\[\pr(\bff(v) \ge 2L)\le 2(M+1)\sum_{\ell = L}^{2L+nM} \exp_2\left(-\frac{|B(v,\lfloor (\ell-2)/M\rfloor-1)|}{5M}\right) \leq 4(M+1)M \exp_2\left( - \frac{(\log n)^{2C'}}{5M} \right) = n^{-\go(1)}\]
given $M \leq (\log n)^{C'}$. By symmetry, we also have $\pr(f(v) \le -2L)=n^{-\go(1)}$, and thus, by the union bound, we obtain that
\[ \pr\left( R(\bff) \geq M \left( C' \frac{\log\log n}{\log(d/2\lambda)} + 2 \right) + 2 \right) \leq n^{-\go(1)}. \]
To prove the bound on $\mathrm{Var}(\bff(v))$, note that $\mathbb{E} \bff(v) = 0$ by symmetry, so with $C=50\ceil{\log(M)/\log(d/2\lambda)}$, we have
\begin{align*}
&\mathrm{Var}(\bff(v)) = \mathbb{E} \bff(v)^2 = \sum_{L = 1}^{nM} L^2 \mathbb{P}(\bff(v) = L) \leq 2 \sum_{L=1}^{nM} L \mathbb{P}(\bff(v) \geq L) \leq 4 \sum_{L=0}^{\floor{nM/2}} 2L \mathbb{P}(\bff(v) \geq 2L) \\
&\leq O\left( C^2 M^2 \right) + \sum_{L=(C+1)M+2}^{nM} 2L \cdot 2(M+1) \sum_{\ell=L}^{2L+nM} \exp_2\left(-\frac{|B(v,\lfloor (\ell-2)/M\rfloor-1)|}{5M}\right) \\
&\leq O\left( C^2 M^2 \right) + 4M \left( (C+1)M+2 \right) \cdot 4(M+1)M \exp_2\left(-\frac{\min\{n/2,(d/2\lambda)^{2C}\}}{5M}\right) \\
&\leq O\left( C^2 M^2 \right)
\end{align*}
where we used \Cref{prop.volume} and $n$ is sufficiently large.
\end{proof}

\section{Final remarks}\label{sec:open}

In this paper, we show that random $M$-Lipschitz functions exhibit small fluctuations on even very weak regular expander graphs, as long as $M$ is not too large as a function of the degree of the graph. Unfortunately, this restriction on $M$ seems difficult to remove using our technique. As mentioned in the introduction, the study of $M$-Lipschitz functions forms a natural bridge between $\Z$-homomorphisms and $\mathbb{R}$-valued Lipschitz functions, and a natural question is whether $\mathbb{R}$-valued Lipschitz functions exhibit small fluctuations on expander graphs (and in what quantitative sense). More precisely, Peled, Samotij, and Yehudayoff~\cite{peled2013lipschitz} asked: do the extensions of Theorem~\ref{MT.nonbip} and Corollary~\ref{thm.flat} hold for $M$ going to infinity independently of $d$ and $n$? One may view the set of $\mathbb{R}$-valued Lipschitz functions (with a one-point boundary condition at $v_0$, say) as a polytope. A possibly related problem (if only in spirit) is the question of estimating/describing the ``metric polytope,'' the set of points $x \in [0,2]^n$ which satisfy the triangle inequality: for all distinct $i,j,k$, $x_i + x_j \geq x_k$. Hypergraph container methods and/or entropy methods were used on this problem; see \cite{Balogh2016, kozma2024does}.

As asked by Peled, Samotij, and Yehudayoff, it is also interesting how weak the condition for $\lambda$ in Theorem~\ref{thm.flat} may be taken --- that is, does \Cref{thm.flat} hold for even $\lambda = (1-\ep)d$? We get closer to answering this question with our condition of $\lambda \le d/5$. We did not try to optimize the upper bound on $\lambda$, but we believe it will be more fruitful to change the definition of expander from~\eqref{def.expander} to spectral-expanders, where all but the largest eigenvalue of the adjacency matrix of the graph are at most $\lambda$ in absolute value.

Bounds on the probability of large fluctuations of Lipschitz functions should be useful in their enumeration. Korsky, Saffat, and Aiylam~\cite{korsky2024lipschitz} discuss enumerating $M$-Lipschitz functions with $M \to \infty$ on two dimensional grids and random graphs. As random graphs are expanders, it is hopeful that more precise enumeration of $M$-Lipschitz functions of random graphs is possible.

Finally, as mentioned in the introduction, our proof techniques extend to $\Z$-homomorphisms with little modification. With suitable generalization, our techniques also extend to general graph homomorphisms to some extent. This will be discussed in \cite{KLP2}. 

\section*{Acknowledgement}

RAK is supported by NSF Graduate Research Fellowship Program Grant No.\ DGE 21-4675. JP is supported by NSF grant DMS-2324978 and a Sloan Fellowship. We thank Louigi Addario-Berry and Wojciech Samotij for a correction to the statement of Theorem~\ref{thm.flat}, and we thank Shayan Gharan for suggesting considering spectral expanders for $\lambda$ very close to $d$.

\end{document}